\documentclass{article}
\usepackage[english]{babel}
\usepackage[T1]{fontenc}
\usepackage[utf8]{inputenc}

\usepackage{geometry}
\usepackage{amssymb,amsmath,amsthm,bbm}
\usepackage{graphicx}
\usepackage[usenames, dvipsnames]{color}
\usepackage{natbib}
\usepackage[colorlinks=true, linkcolor=blue, citecolor=blue]{hyperref}
\usepackage{subfigure}
\usepackage{mathrsfs}

\renewcommand{\d}{ {\rm d}}

\def\twofig{.48\textwidth}

\newcommand{\IND}{\ensuremath{\mathbbm{1}}}
\def\cH{\mathcal{H}}

\DeclareUrlCommand\email{\urlstyle{rm}}

\numberwithin{equation}{section}

\newtheorem{theorem}{Theorem}[section]

\newtheorem{corollary}[theorem]{Corollary}
\newtheorem{proposition}[theorem]{Proposition}

\theoremstyle{definition}

\newtheorem{Remark}[theorem]{Remark}
\newenvironment{remark}{\begin{Remark}\rm}{\end{Remark}}

\newtheorem{Example}[theorem]{Example}

\newcommand{\eq}{\begin{equation}}
\newcommand{\qe}{\end{equation}}
\newcommand{\beq}{\begin{equation}}
\newcommand{\eeq}{\end{equation}}

\newcommand{\N}{\mathbb{N}}
\newcommand{\Z}{\mathbb{Z}}
\newcommand{\R}{\mathbb{R}}
\newcommand{\C}{\mathbb{C}}
\newcommand{\E}{\mathbb{E}}
\renewcommand{\P}{\mathbb{P}}

\newcommand{\bs}{\boldsymbol}

\newcommand{\Var}{\mathbb{V}{\mathrm{ar}}}

\renewcommand{\epsilon}{ \varepsilon}
\renewcommand{\phi}{ \varphi}
\renewcommand{\d}{ {\rm d}}
\renewcommand{\frak}{\mathfrak}

\renewcommand{\Im}{ \frak{Im}}

\newcommand{\e}{\mathrm{e}}
\renewcommand{\i}{\mathrm{i}}

\renewcommand{\H}{\mathcal{H}}
\newcommand{\D}{\mathbb{D}}

\newcommand{\law}{ \stackrel{\mathrm{law}}{=}}

\renewcommand{\L}{\mathscr L}

\newcommand{\GAF}{\mathsf{GAF}}

\title{Time-frequency transforms of white noises\\
and Gaussian analytic functions}
\date{}
\begin{document}
\author{
\ R{\'e}mi Bardenet\,\footnote{Universit\'e de Lille, CNRS, Centrale Lille, UMR 9189 -- CRIStAL, 59651 Villeneuve d’Ascq Cedex, France. Email: \href{mailto:remi.bardenet@gmail.com}{\nolinkurl{remi.bardenet@gmail.com}}}\;\quad
\ Adrien Hardy\,\footnote{Universit\'e de Lille, CNRS, Inria, UMR 8524, Laboratoire Paul Painlevé, F-59000 Lille, France.  \newline
Email: \href{mailto:adrien.hardy@univ-lille.fr}{\nolinkurl{adrien.hardy@univ-lille.fr}},\,
}
}
\maketitle

\begin{abstract} A family of Gaussian analytic functions (GAFs) has recently been linked to the Gabor transform of white Gaussian noise \citep{BaFlCh18}. This answered pioneering work by \cite{Fla15}, who observed that the zeros of the Gabor transform of white noise had a very regular distribution and proposed filtering algorithms based on the zeros of a spectrogram. The mathematical link with GAFs provides a wealth of probabilistic results to inform the design of such signal processing procedures. In this paper, we study in a systematic way the link between GAFs and a class of time-frequency transforms of Gaussian white noises on Hilbert spaces of signals. Our main observation is a conceptual correspondence between pairs (transform, GAF) and generating functions for classical orthogonal polynomials. This correspondence covers some classical time-frequency transforms, such as the Gabor transform and the Daubechies-Paul analytic wavelet transform. It also unveils new windowed discrete Fourier transforms, which map white noises to fundamental GAFs.  All these transforms may thus be of interest to the research program ``filtering with zeros''. We also identify the GAF whose zeros are the extrema of the Bargmann transform of the white noise and derive their first intensity. Moreover, we discuss  subtleties in defining a white noise and its transform on infinite dimensional Hilbert spaces. Finally, we provide quantitative estimates concerning the finite-dimensional approximations of these white noises, which is of practical interest when it comes to implementing signal processing algorithms based on GAFs.
\end{abstract}
\tableofcontents

\section{Introduction}
\label{s:intro}
Spectrograms are a cornerstone of time-frequency analysis. They associate to each time and frequency a real number that measures the energy content of a signal at that time and frequency, unlike global-in-time tools such as the Fourier
transform. Since it is natural to expect that there is more energy where there
is more information or signal, most methodologies have focused on detecting and
processing the local maxima of the spectrogram, see the reference books \citep{Coh95,Fla98,Gro01}. Usual techniques include \emph{ridge extraction},
e.g., to identify chirps, or \emph{reassignment} and \emph{synchrosqueezing}, to
better localize the maxima of the spectrogram before further quantitative
analysis.

\cite{Fla15} has recently observed that, when applied to Gaussian white noise, the short-time Fourier transform (STFT) with Gaussian window (or \emph{Gabor transform}; \cite[Chapter 3]{Gro01}) is an analytic function, and thus has well-identified, separated zeros. \cite{Fla15} proposed to use the point pattern formed by these random zeros in filtering and reconstruction of signals in noise. In particular, \cite{Fla15} empirically assessed that the zeros of the Gabor spectrogram of white noise spread out very evenly on the time-frequency plane, with very regular Voronoi tesselations.

In the same setting of a short-time Fourier transform with Gaussian window, \cite*{BaFlCh18} proved that the zeros of the spectrogram of white noise have the same distribution as the zeros of the \emph{planar Gaussian analytic function} (GAF). This provided a first link between zeros of random spectrograms and GAFs, which explained the observations of \cite{Fla15}. \cite{BaFlCh18} then used probabilistic results on this planar GAF to inform the design of signal reconstruction procedures in the spirit of \citep{Fla15}.

GAFs are Gaussian processes that are supported on analytic functions. They have generated recent interest among probabilists, among other reasons due to their connection to determinantal point processes, see \citep{HKPV09}.
 Among the sets of zeros of these random analytic functions, three stand out as the only zero sets of GAFs that are respectively invariant to isometries of the plane $\C$, the hyperbolic plane $\mathbb H$, and the sphere $\mathbb S$. Following the result of \cite{BaFlCh18}, a natural question is whether other common transforms in signal processing (say, an STFT with a non-Gaussian window or a wavelet transform) also tie to GAFs. We answer here this question affirmatively. Conversely, since the GAF identified by \cite{BaFlCh18} is one of the three with special invariance properties, a probabilist may ask whether the other two also naturally appear. We also give here an affirmative answer. In particular, the analytic wavelet transform of \cite{DaPa88} is shown to be related to the so-called hyperbolic GAF, with a set of zeros that is a determinantal point process. The latter link has also been simultaneously and independently made by \cite*{AHKR18Sub}.

More generally, we establish a {\it template} theorem that links a generic transform to a GAF. The central observation of this work is that there is a correspondence between pairs (transform, GAF) and generating functions for classical orthogonal polynomials. This allows us to start a dictionary of such pairs, which includes some classical signal processing transforms and GAFs. The type of transform we consider can be loosely described as taking a signal in a complex Hilbert space $\cH$, decomposing this signal onto an orthonormal basis $(f_k)$ of $\cH$ associated with a family of classical orthogonal polynomials, and reconstructing a complex function by replacing each element of the basis by a  complex monomial, up to an appropriate isometry of $\C,\mathbb H$ or $\mathbb S.$ We shall see that the zeros of the associated white noise's transform have the same distribution than the zeros of the GAF associated with the corresponding space $\C/\mathbb H/\mathbb S$.

We recover some known transforms, such as the STFT with Gaussian window as in \citep{BaFlCh18} or the analytic wavelet transform of \cite{DaPa88}, but also unveil new transforms as well as new GAFs. Among these new transforms, some are \emph{discrete} windowed Fourier transforms that preserve {exactly} the distribution of the zeros of their continuous counterparts, a fact of independent computational interest. {Our main results are summarized in Table~\ref{f:summary}.
As we shall see below, the  transform acting on $L^2(\R,\C)$ is the Bargmann transform, and is well-known to be related to the STFT with a Gaussian window or the Gabor transform \citep{Gro01}. The transform acting on the Hardy space $H^2(\R)$ of analytic signals is the Bergman transform, and is a similar modification of the Daubechies-Paul wavelet transform \citep{Abr08}. The three other transforms, acting on discrete signals, are new. We have picked them for to their remarkable interaction with GAFs and classical orthogonal polynomials; it is somehow intriguing that discrete transforms also lead to the same GAFs. We hope they can be of interested to the signal processing community.}

\begin{table}[h]
  {
\centering
\begin{tabular}{|c|c|c|c|c|}
  \hline
$\H$  & Transform & OP & GAF & Theorem\\
  \hline
  $L^2(\R,\C)$ & $\displaystyle \frac{\e^{-z^2/2}}{\pi^{1/4}} \int_\R \overline{f(x)}\,\e^{\sqrt2 xz-x^2/2}\d x$ &  Hermite & $\C$ & Th.~\ref{t:planarL} \\
    \hline
  $\ell^2(\N,\C)$ & $\displaystyle\sum_{x\in\N}\overline{f(x)} \frac{z^x }{\sqrt{x!}}$  & Charlier 	& $\C$ &Th.~\ref{t:planarl}  \\
    \hline
  $H^2(\R)$ & $\displaystyle  \frac1{(1-z)^{2\beta+1}} \int_{\R_+} \overline{\hat f(x)}\, x^{\beta}\e^{-\frac x 2\tfrac{1+z}{1-z}} \d x$ &  Laguerre & $\mathbb H$ & Th.~\ref{t:Hyp}  \\
  \hline
  $\ell^2(\N,\C)$ & $\displaystyle\sum_{x\in\N} \overline{f(x)} \sqrt{\frac{\Gamma(x+\alpha+1)}{x!}}\, z^x$ & Meixner & $\mathbb H$ & Th.~\ref{t:hyp} \\
  \hline
    $\C^{N+1}$ & $\displaystyle \sum_{x=0}^N \overline{f(x)} \sqrt{ {N \choose x}} \,z^x$ & Krawtchouk & $\mathbb S$ &Th.~\ref{t:spherical}  \\
  \hline
\end{tabular}
\caption{Examples of correspondences between a Hilbert space $\H$ of signals, a transform, a family of orthogonal polynomials, a Gaussian analytic function, and the corresponding theorem in this work.}
\label{f:summary}
}
\end{table}

{
An important conceptual problem related to the previous discussion is the rigorous definition of the white noises to which we apply our transforms. Another motivation of this work was to bring forward the subtleties one can meet in constructing such an object. We use the abstract Wiener space approach of \cite{Gro67}. This construction requires specifying an orthonormal basis, which allows a description of where the white noise actually lives. Indeed,  constructions like the Bochner-Minlos theorem or abstract Gaussian spaces typically yield measures supported on the space of tempered distributions, which is large compared to the original space of signals. The approach of \cite{Gro67} allows to build a white noise living on a smaller space of distributions, which are quite regular (typically the dual of a Sobolev space of order $1$ or $2$ having some low moment conditions). In particular, we obtain a natural definition of an analytic Gaussian white noise, a notion that has generated some confusion in signal processing; this echoes the work of \cite{Pug82} and answers a question raised in \cite[Section 4.4]{BaFlCh18}. Finally, our tailored construction of white noise allows us to provide quantitative estimates on the loss of accuracy incurred by replacing the white noise by its finite dimensional approximations. This yields control on how the distribution of the zeros of a transform of white noise departs from the zeros of the truncated transforms that are implemented in practice.
}

The rest of the paper is organized as follows: In Section~\ref{s:gaf} we give an introduction to GAFs and state our main results in the form of several theorems linking known and new transforms of white noises to GAFs. Next, in Section~\ref{s:white}, we discuss the definition of a white noise on an infinite-dimensional space and its subtleties, and we give the rigorous definition used in this work. We also identify a common structure in all the transforms of interest in this work, and we justify that the transforms of white noise are well defined in this context. The section ends with our template theorem, of which all the theorems in Section~\ref{s:gaf} are corollaries once combined with appropriate orthogonal polynomial generating functions identities. Section~\ref{s:proofs} then contains all detailed proofs. In Section~\ref{s:computation}, we give computational details. We further derive several results on finite dimensional approximations of the white noises and their transforms, including their zeros, which are of practical interest. Finally, we provide a companion Python toolbox\footnote{\url{https://github.com/rbardenet/tf-transforms-and-gafs}} to reproduce all figures in the paper.

\paragraph{Acknowledgments:} This work has been partially supported by ANR JCJC \textsc{BoB} (ANR-16-CE23-0003) and Labex \textsc{CEMPI} (ANR-11-LABX-0007-01).

\section{GAFs and time-frequency transforms}
\label{s:gaf}

In the following, $ N_\C(0,\sigma^2)$ refers to the law of a centered Gaussian variable on $\C$ with variance $\sigma^2$, namely the law of $Z:=\sigma(X+\i Y)/\sqrt 2$ with $X,Y$ identically distributed and independent (i.i.d.) real standard $ N(0,1)$ variables; note that the $\sqrt 2$ is here to have $\Var(Z):=\E|Z|^2-|\E(Z)|^2=\sigma^2$. More generally, given $m\geq 1$, a complex Gaussian vector of mean $\mu\in\C^m$ and covariance matrix $\Sigma:=AA^*$, with $A\in\mathcal M_{n}(\C)$, is a random variable with the same distribution as $AX+\mu$, where $Z:=(Z_1,\ldots,Z_d)$ with $(Z_i)_{i=1}^m$ i.i.d. $N_\C(0,1)$ random variables. Given an open subset  $\Lambda\subset\C$, we let $\mathcal A(\Lambda)$ be the space of analytic functions on $\Lambda$.

\subsection{Gaussian analytic functions and their zeros}
\label{s:GAFsAndTheirZeros}

A \emph{Gaussian analytic function} (GAF) on $\Lambda\subset\C$ is a random variable $\GAF$ taking its values in $\mathcal A(\Lambda)$ such that, for any $m\geq 1$ and $z_1,\ldots, z_m\in\Lambda$, the vector $(\GAF(z_1),\ldots,\GAF(z_m))$ is a complex Gaussian vector. In this work, we focus on GAFs of the form
\eq
\label{canGAF}
\GAF(z)=\sum_{k=0}^\infty \xi_k\, \Psi_k(z)
\qe
where $(\xi_k)_{k\in\N}$ is a sequence of i.i.d. $N_\C(0,1)$ random variables and $(\Psi_k)_{k\in\N}$ is a sequence in $\mathcal A (\Lambda)$ satisfying
\eq
\label{CKweak}
\sup_{z\in K}\sum_{k=0}^\infty | \Psi_k(z)|^2<\infty
\qe
for any compact $K\subset\Lambda$. Condition \eqref{CKweak} ensures that $\GAF$ converges almost surely on every compact subset of $\Lambda$, and thus that $\GAF$ in \eqref{canGAF} is a well-defined analytic function, see \cite[Lemma 2.2.3]{HKPV09}. As a centered Gaussian process on the space of analytic functions $\mathcal A(\Lambda)$, $\GAF$ is completely characterized by its covariance structure, namely
\eq
\label{e:cov}
\E\Big[\GAF(z)\overline{\GAF(w)}\Big]=\sum_{k=0}^\infty \Psi_k(z)\overline{\Psi_k(w)},\qquad z,w\in\Lambda.
\qe
Our particular interest lies in the random set $\GAF^{-1}(0)$ of zeros of $\GAF$. More precisely, our interest lies in the \emph{random} zeros of $\GAF$, that is, zeros that are not present in every sample of $\GAF$. \cite[Lemma 2.4.1]{HKPV09} states that the random zeros of any nonzero GAF are simple with probability one. Thus, the random zeros form a (simple) point process, namely a random variable taking values in the locally finite subsets of $\Lambda$. It is often described through its \emph{joint intensities} $(\rho_k)_{k\geq 1}$, defined by
\eq
\int_{\Lambda^k}\psi(z_1,\ldots,z_k)\,\d\rho_k(z_1,\ldots, z_k) = \E\left[\sum_{\substack{(x_1,\ldots,x_k)\in \GAF^{-1}(0)\\x_1\neq \cdots\neq x_k}} \psi(x_1,\ldots,x_k)\right]
\label{e:correlationIntensities}
\qe
for any bounded compactly supported measurable map $\psi:\Lambda^k\to\R$.

It is also known that the law of the point process $\GAF^{-1}(0)$ is characterized by its first intensity $\rho_1$ solely and moreover that, if two GAFs, say $\GAF$ and $\widetilde{\GAF}$, share the same zero set distribution, namely $\GAF^{-1}(0)\law \widetilde{\GAF}^{-1}(0)$, then there exists a non-vanishing and deterministic $\phi\in\mathcal A(\Lambda)$ such that $\GAF=\phi \times \widetilde{\GAF}$, see \cite[Theorem 2.5.2]{HKPV09}.

\subsection{Three GAFs, five time-frequency transforms}
 There are three prototypical spaces in complex geometry: the complex plane $\C$, the hyperbolic plane $\mathbb H$, and the sphere $\mathbb S$. For instance, the so-called uniformization theorem says that any simply connected Riemann surface is conformally equivalent to $\C$,  $\mathbb H$ or $\mathbb S$.  The complex plane is a Riemannian manifold with null curvature when equipped with the Euclidean metric $g_\mathrm{flat}$. The hyperbolic plane $\mathbb H$ can be modeled as the unit disc $\mathbb D:=\{z\in\C:\,|z|< 1\}$ with the hyperbolic metric $(1-|z|^2)^{-2}g_\mathrm{flat}$, giving a Riemannian manifold of constant negative curvature. It can alternatively be modeled as the upper-half plane $\C_+:=\{z\in\C:\, \Im(z)>0\}$ with the metric $\Im(z)^{-1}g_\mathrm{flat}.$ The sphere $\mathbb S$ can be modeled as the extended complex plane $\C\cup\{\infty\}$ with the spherical metric $(1+|z|^2)^{-2}g_\mathrm{flat}$, yielding a Riemannian manifold of constant positive curvature. This metric is the image by the stereographic projection of the natural metric on the unit sphere of $\R^3$ induced by the Euclidean structure of $\R^3$. The group of isometries for each of these three spaces is well-known, and it is natural to ask which GAFs have their zeros invariant under these isometries. It turns out that, up to multiplication by a non-vanishing and deterministic analytic function, there is only a one-parameter family of GAFs satisfying this condition for each space; see \citep{HKPV09}.

 A realization of the zero sets of each of the three invariant GAFs is shown in Figure~\ref{f:gafs} as white dots. For readers unfamiliar with complex geometry, Figure~\ref{f:gafs} can also serve as an illustration of the previous paragraph. Roughly speaking, on all three subfigures, the white dots are uniformly spread, with roughly equal distance between each pair of nearest neighbors. Figure~\ref{f:planarDemo} represents a portion of the complex plane, and the flat metric corresponds to the intuitive idea of a uniform distribution of zeros on this portion. Figure~\ref{f:hyperbolicDemo} represents a portion of $\C_+$ with the hyperbolic metric $\Im(z)^{-1}g_\mathrm{flat}.$ Again, the white dots are uniformly spread for this metric, with distances between nearest neighbors roughly equal. Figure~\ref{f:sphericalDemo} similarly represents the sphere with the spherical metric.

\begin{figure}
\centering
\subfigure[Zeros of the planar GAF]{
  \includegraphics[width=\twofig]{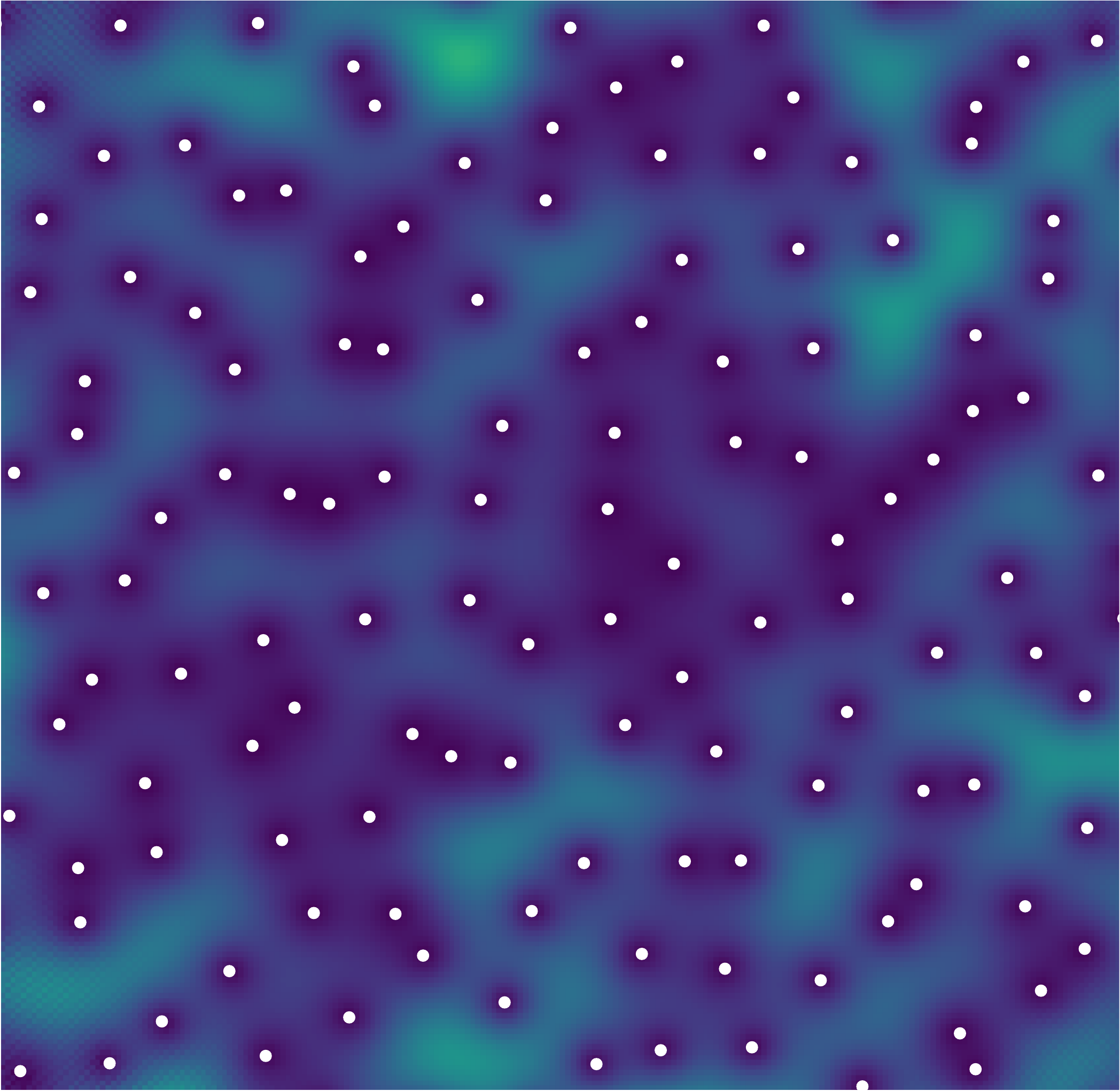}
  \label{f:planarDemo}
  }
\subfigure[Zeros of the spherical GAF]{
  \includegraphics[width=\twofig]{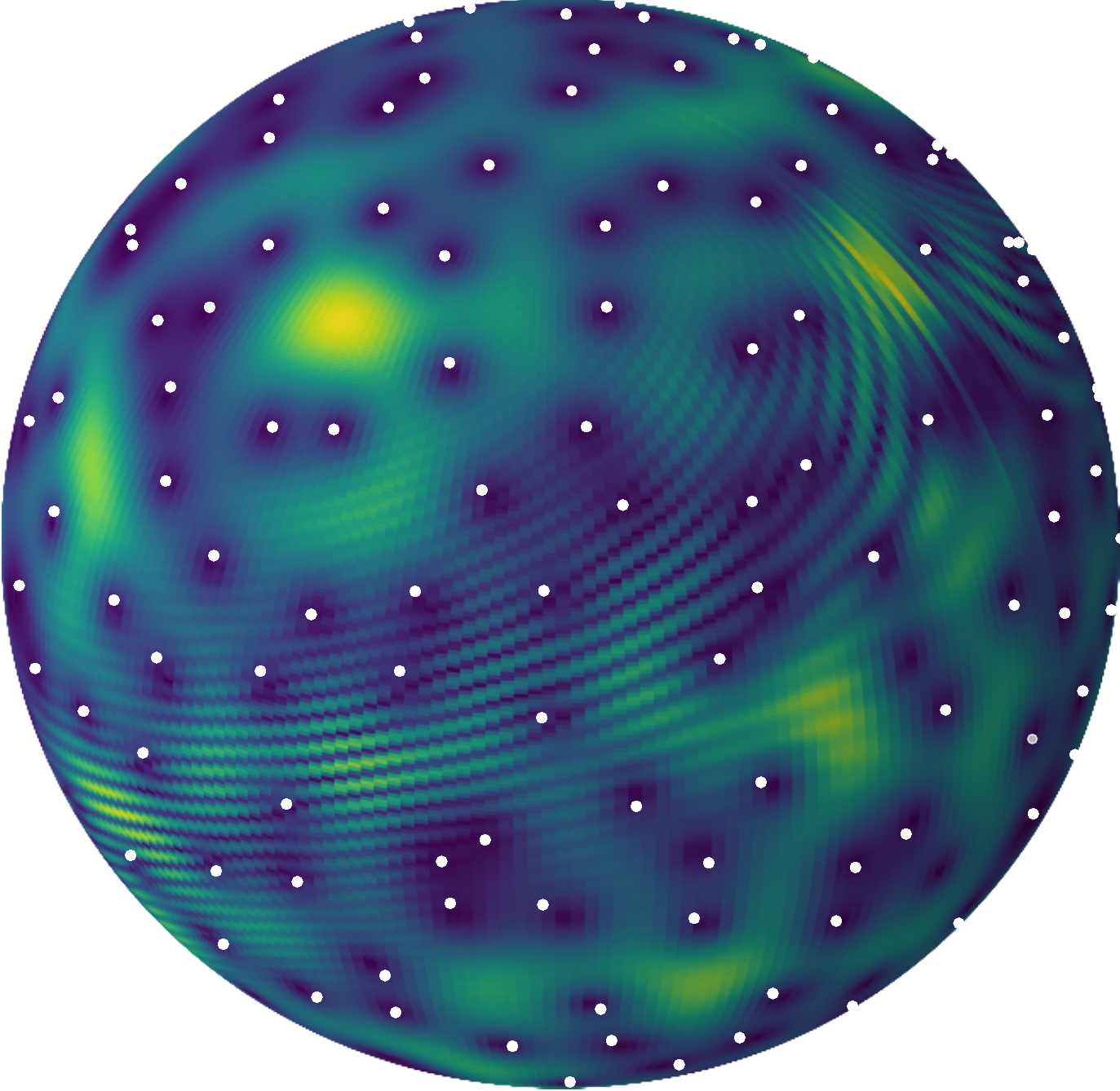}
  \label{f:sphericalDemo}
  }
\subfigure[Zeros of the hyperbolic GAF]{
  \includegraphics[width=\textwidth]{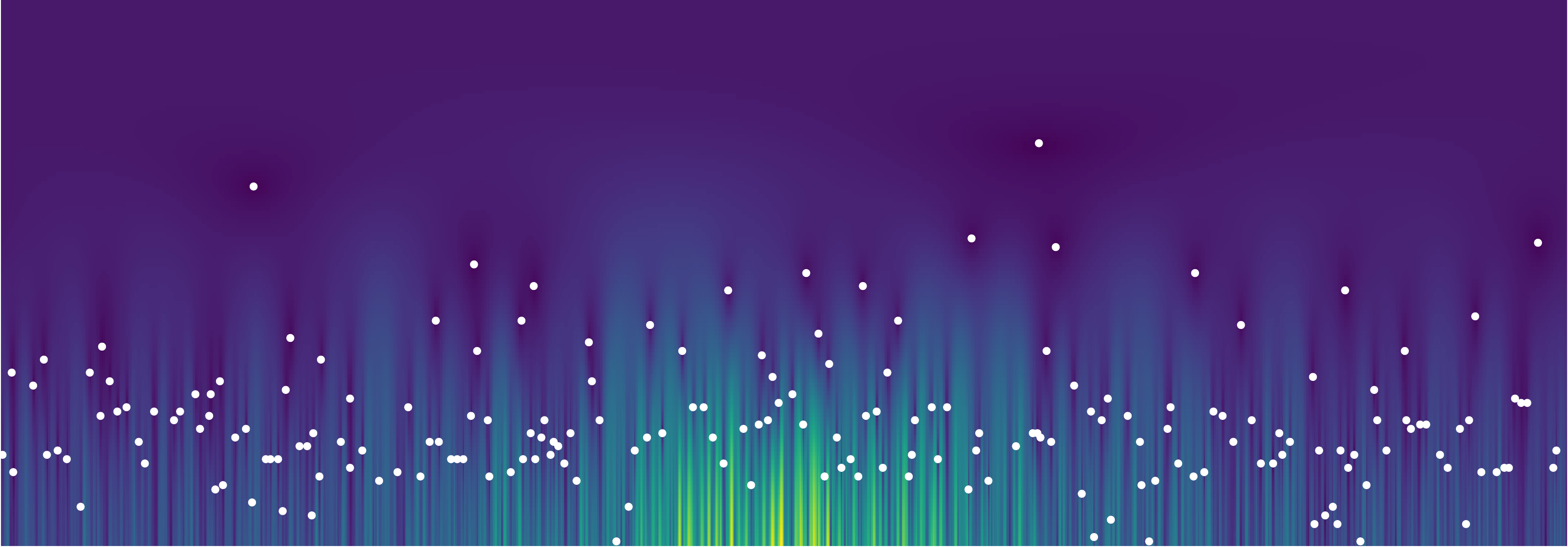}
  \label{f:hyperbolicDemo}
  }
\caption{White dots are sample zero sets of \ref{f:planarDemo} the planar GAF with $\ell=1$, obtained by truncating the Charlier transform $\ell_\mathbb{C}^{(1)}\xi$, \ref{f:sphericalDemo} the spherical GAF with $N=2^8$, obtained by the Krawtchouk transform $\ell_{\mathbb S}^{(N)}\xi$, and \ref{f:hyperbolicDemo} the hyperbolic GAF with $\alpha=0$, obtained by a Riemann approximation to the Paul-Daubechies transform $W_0$ of the truncated series of an analytic white noise.}
\label{f:gafs}
\end{figure}

\subsubsection{The planar GAF}
\label{s:planar}
Let us take $\Lambda=\mathbb{C}$. First, the \emph{planar GAF} of parameter $\ell>0$,  defined on $\mathbb{C}$ by
\eq
\label{planarGAF}
\GAF_{\C}^{(\ell)}(z):=\sum_{k=0}^\infty \xi_k \,\frac1{\sqrt{ k!}}\left(\frac z \ell\right)^k ,
\qe
where $(\xi_k)_{k\in\N}$ is a sequence of i.i.d. $N_\C(0,1)$ random variables, has zeros invariant under the isometries of $(\C,g_\mathrm{flat})$, namely the mappings $z\mapsto a z+b$ with $a,b\in\C$ and $|a|=1$. In other words, the distribution of the zeros of \eqref{planarGAF} is invariant to translations and rotations. The covariance kernel of \eqref{planarGAF} is given by
\eq
\label{e:covPlanar}
\E\Big[\GAF_{\C}^{(\ell)}(z)\overline{\GAF_{\C}^{(\ell)}(w)}\,\Big]=\e^{z\overline  w/\ell^2}.
\qe

\paragraph{The STFT with Gaussian window and Hermite polynomials.}
A first connection between time-frequency transforms of white noises and GAFs has been obtained by \cite*{BaFlCh18}. More precisely, for any complex-valued signal $f\in L^2(\R,\C)$, define the \emph{Bargmann transform} as
\eq
\label{barg}
\L_{\mathrm{\C}} f(z):=\frac{\e^{-z^2/2}}{\pi^{1/4}} \int_\R \overline{f(x)}\,\e^{\sqrt2 xz-x^2/2}\d x,\qquad z\in\C\,.
\qe
The Bargmann transform is known to be the STFT with Gaussian window up to nonvanishing multiplicative terms after reflection $z \mapsto \overline z$, see e.g. \cite[Proposition 3.4.1]{Gro01}.
The time-frequency interpretation of the STFT is standard, see e.g. \citep{Coh95,Fla98,Gro01}. If $f_k$ stands for the $k$-th Hermite function, see \eqref{e:hermiteFunctions} for a precise definition, then an important property of the Bargmann transform is that
$$
\L_{\mathrm{\C}} f_k(z)=\frac{z^k}{\sqrt k!}.
$$
Thus $\L_{\mathrm{\C}}$ maps the Hermite functions to monomials. From a time-frequency perspective, this translates into the spectrogram of a Hermite function being well localized, essentially supported on a centered ring in the time-frequency plane. This makes Hermite functions amenable to postprocessing procedures like reassignment \citep{Fla13}.

The following prototypical result is the main motivation for this work. The precise definition of white noise we use in this paper is based on the abstract Wiener spaces of \cite{Gro67} and will be developed in Section~\ref{s:white}. For now we simply note that a white noise depends by construction on a specific basis of the Hilbert space of signals of interest.

\begin{theorem}
\label{t:planarL}
Let $\xi$ be the white noise on $L^2(\R,\C)$ associated with the Hermite basis $(f_k)$, as defined in Proposition \ref{p:white}. Then,
$$
\L_{\mathrm{\C}}\xi \law \GAF_\C^{(1)}.
$$
In particular, the zeros of $\L_{\mathrm{\C}}\xi$, which have the same law as the zeros of the STFT with Gaussian window of $\xi$, are invariant under the planar isometries $z\mapsto \alpha z+\beta$ with $\alpha,\beta\in\C$ and $|\alpha|=1.$
\end{theorem}

Theorem~\ref{t:planarL} has already been obtained in \citep{BaFlCh18} with a slightly different definition of the white noise, taking its values in the space of tempered distributions. We provide in Section~\ref{e:gabor} an alternative proof which fits in our more general setting. We note that \cite{BaFlCh18} crucially make use of the invariance to the planar isometries to estimate functional summary statistics from realizations of the zeros, yielding detection and reconstruction algorithms that formalize the original approach of \cite{Fla15}.

Next, we ask about a discrete counterpart of the STFT which also maps the white noise to the planar GAF.

\paragraph{A discrete STFT and Charlier polynomials.}
Let $\N:=\{0,1,2,\ldots\}$. Having in mind \eqref{planarGAF}, a natural candidate is  the transform given by,
 for any $f\in \ell^2(\N,\C)$  and $z\in\C$,
\eq
\label{e:charlierTransform}
\ell_\mathbb{C}f(z) := \sum_{x\in\N}\overline{f(x)} \frac{z^x }{\sqrt{x!}}\,.
\qe

Indeed, it is a direct consequence of the general Theorem \ref{t:main} that, if $\xi$ is a white noise associated with the canonical basis $f_k(\ell):=\delta_{k\ell}$, then $\ell_\mathbb{C}\xi= \GAF_\C^{(1)}$ in law. Moreover, as for the Bargmann transform, $\ell_\C$ maps the canonical basis $f_k$ to the monomials $z^k/\sqrt{k!}$. From a signal processing perspective, this means that signals perfectly localized in time --members of the canonical basis-- will map to monomials. In the STFT case, the signals that mapped to monomials were the Hermite functions, which are better localized in frequency. An intuitive equivalent in the discrete case would be to replace Hermite polynomials by discrete orthogonal polynomials, which would play the role of elementary waveforms. Our next result shows that the role of the Hermite functions can indeed be played by Charlier functions.

More precisely, for $a>0$, let $(f_k)$ be the Charlier functions \eqref{e:Charlierfunction} coming with the Charlier polynomials $C_k^{(a)}$, that is, the discrete orthogonal polynomials with respect to the Poisson distribution of parameter $a$. The precise definition of Charlier functions is given later in \eqref{orthoCharlier}, but to gain intuition, we plot in Figure~\ref{f:charlier} the first Charlier functions $(f_k)_{k=0,\dots,4}$ for $a=5$. The functions oscillate more as $k$ grows, as expected.


\begin{theorem}
\label{t:planarl}
If $\xi$ is the white noise on $\ell^2(\N,\C)$ associated with the Charlier functions $(f_k)$ of parameter $a>0$, as defined in Proposition \ref{p:white}, then
$$
\ell_{\mathrm{\C}}\xi \law  \GAF_\C^{(1)}.
$$
In particular, the zeros of $\ell_{\C}\xi$ have the same law as the zeros of $\GAF_\C^{(1)}$ and this law is independent of $a$. Moreover, for any $k\in\N$,
we have 
$$
\ell_{\mathrm{\C}}f_k(z)=\e^{a/2-\sqrt a z}\frac{(\sqrt a-z)^k}{\sqrt{k!}}.
$$
\end{theorem}
Theorem \ref{t:planarl} is proven in Section \ref{sec:proofCharlier}.

For now, let us interpret the transform $\ell_\C$ from a time-frequency perspective: taking $z=r\e^{i\theta}$, the definition becomes
\beq
\ell_\mathbb{C} f(z) = \sum_{x\in\N}\overline{f(x)}\frac{r^x}{\sqrt{x!}} \,\e^{ix\theta},
\label{e:cmp}
\eeq
which is a short-time discrete Fourier transform (STDFT), amenable to fast Fourier transforms. The window of this STDFT is the so-called Conway--Maxwell--Poisson distribution of parameters $r$ and $1/2$. This is a unimodal distribution, akin to an overdispersed Poisson distribution, with mode $\lfloor r^{2}\rfloor$. In the end, $r$ in \eqref{e:cmp} plays the role of the mean of the shifted Gaussian in the STFT with Gaussian window, and the squared modulus of \eqref{e:cmp}, as a function of ``time'' $r$ and ``frequency'' $\theta$, is a spectrogram with a window sliding quadratically with time. An additional difference with the classical STFT is that the time-frequency interpretation of the complex plane is now thought in terms of polar coordinates.

Figure~\ref{f:planarDemo} shows a realization of the zeros of the planar GAF, approximated, for each $r$, by truncating the series \eqref{e:cmp} and computing the corresponding discrete Fourier transform. The zeros are \emph{very} uniformly spread across the complex plane, as if repelling each other, a phenomenon that has several probabilistic consequences, see e.g. \cite{HKPV09}. In Figure~\ref{f:planarDemo_detection}, we show the effect of adding a Charlier function with a large real coefficient to a complex white noise before taking the transform. The time-frequency support of the signal is a centered annulus, as expected from Theorem~\ref{t:planarl}, thus perfectly mimicking the behavior of Hermite functions and the Gabor transform \citep{Fla13}. Zeros are pushed away from the support of the signal, an observation which is at the basis of the reconstruction approaches in \citep{Fla15,BaFlCh18}. Note how a few zeros are trapped inside the annulus.

\begin{figure}
\centering
\subfigure[Charlier functions]{
  \includegraphics[width=\twofig]{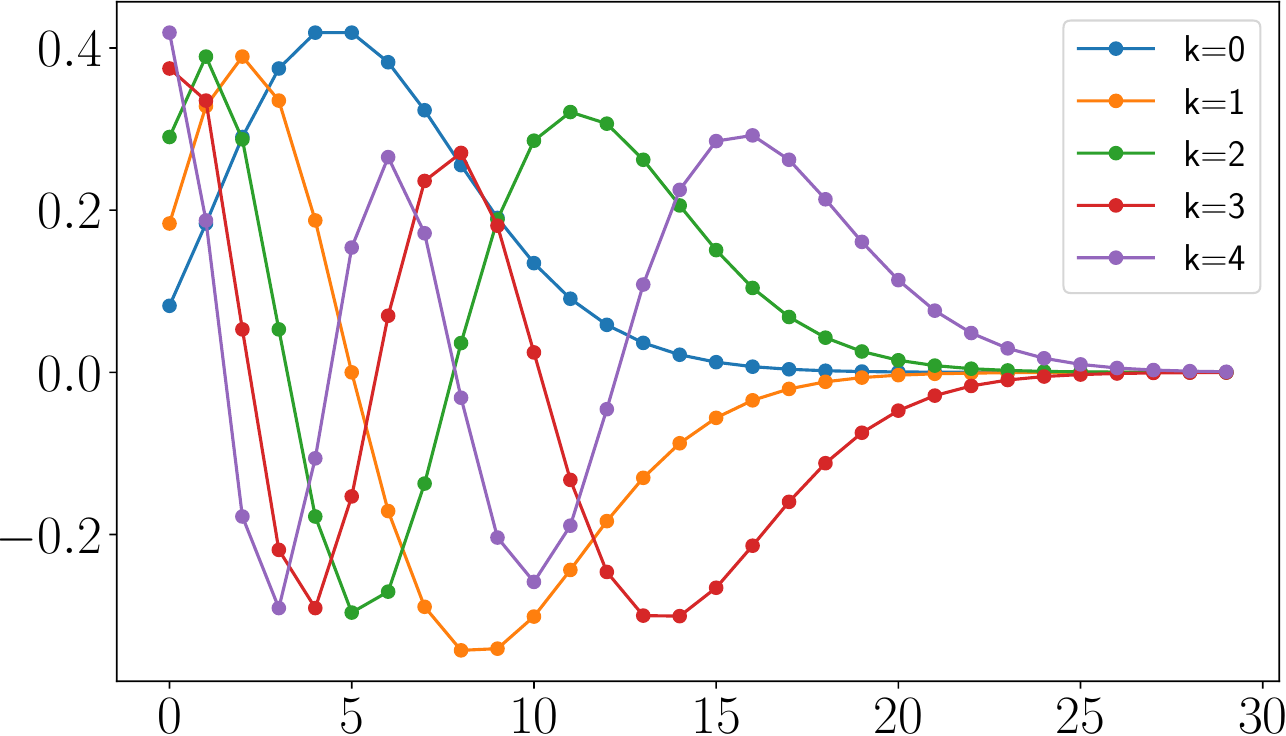}
  \label{f:charlier}
  }
\subfigure[Krawtchouk functions]{
  \includegraphics[width=\twofig]{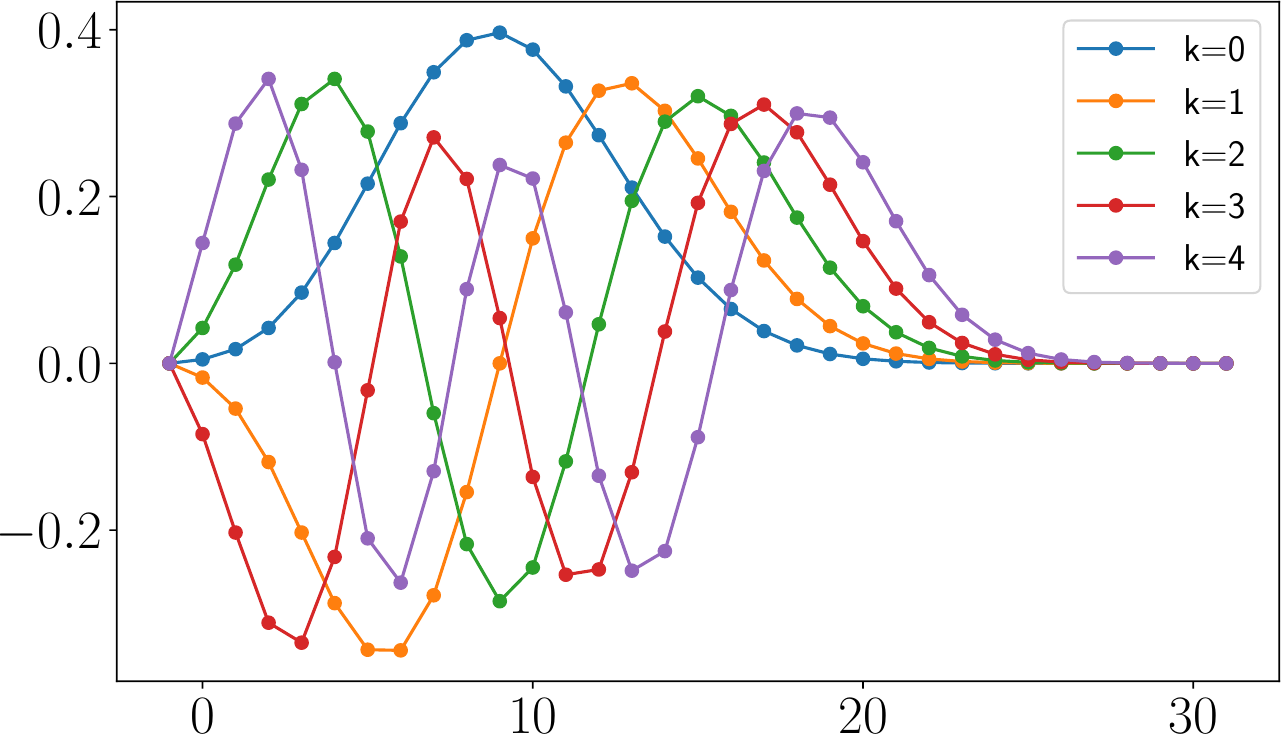}
  \label{f:kravchuk}
  }\\
\subfigure[Inverse Laguerre functions]{
  \includegraphics[width=\twofig]{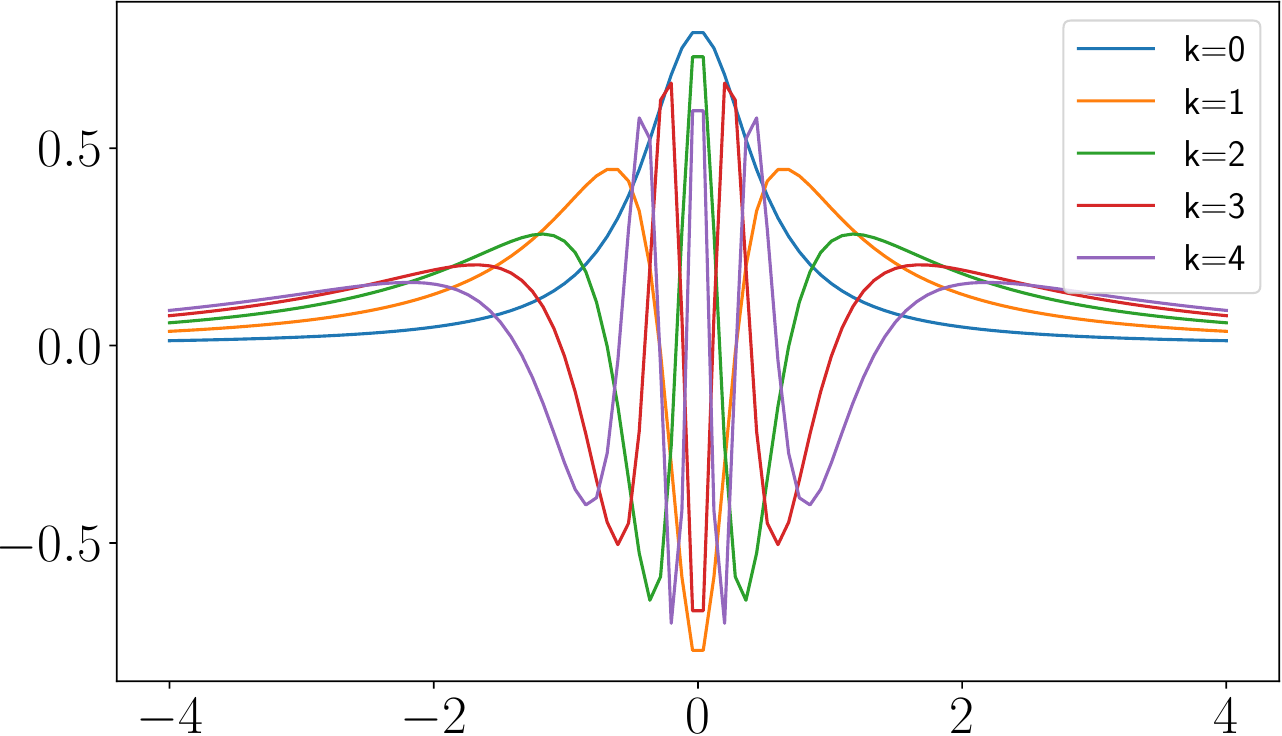}
  \label{f:inverseLaguerre}
  }
\caption{The first few terms of some of the orthonormal bases used in this paper.}
\end{figure}

\begin{figure}
\centering
\subfigure[Charlier]{
  \includegraphics[width=\twofig]{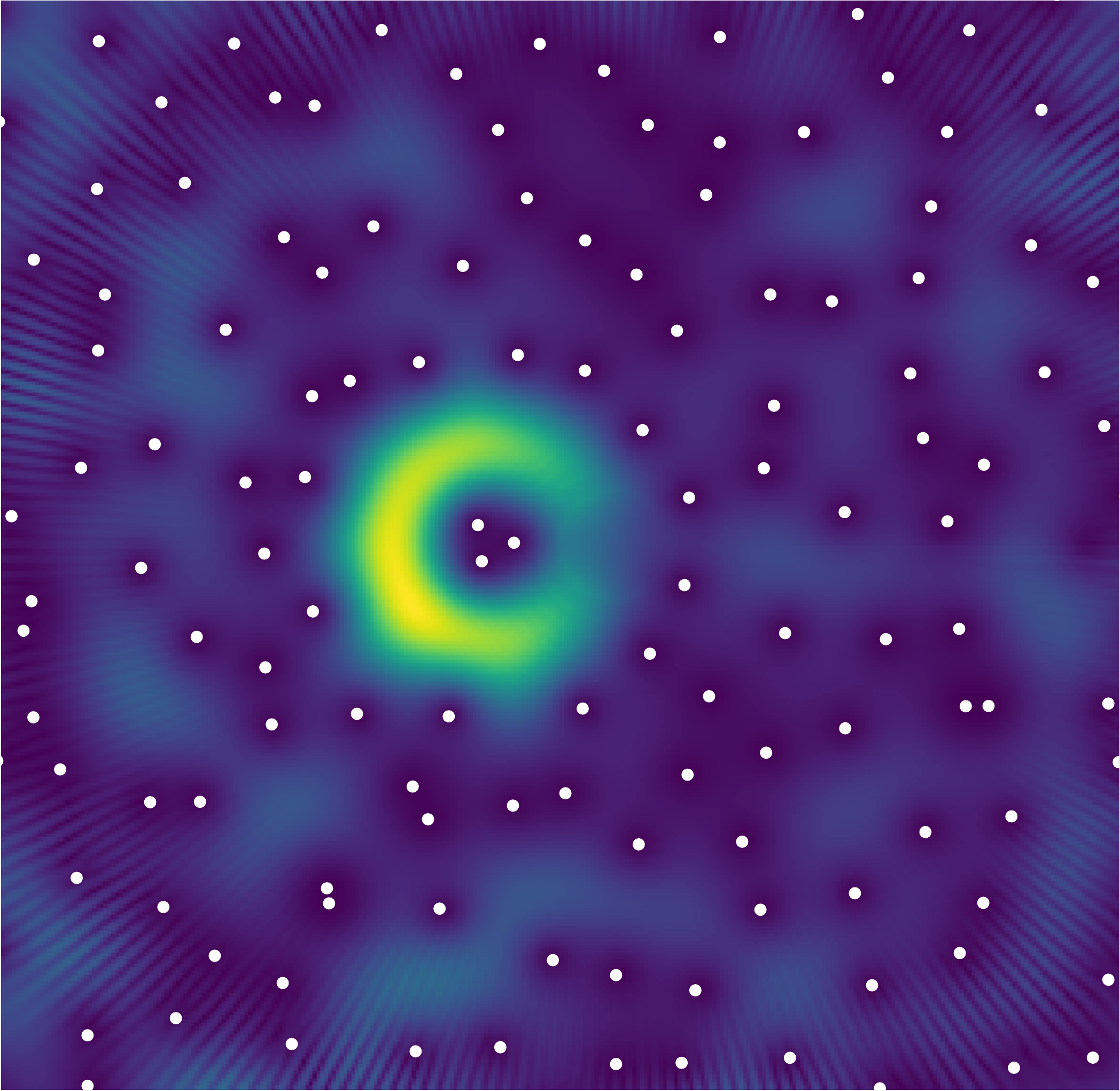}
  \label{f:planarDemo_detection}
  }
\subfigure[Krawtchouk]{
  \includegraphics[width=\twofig]{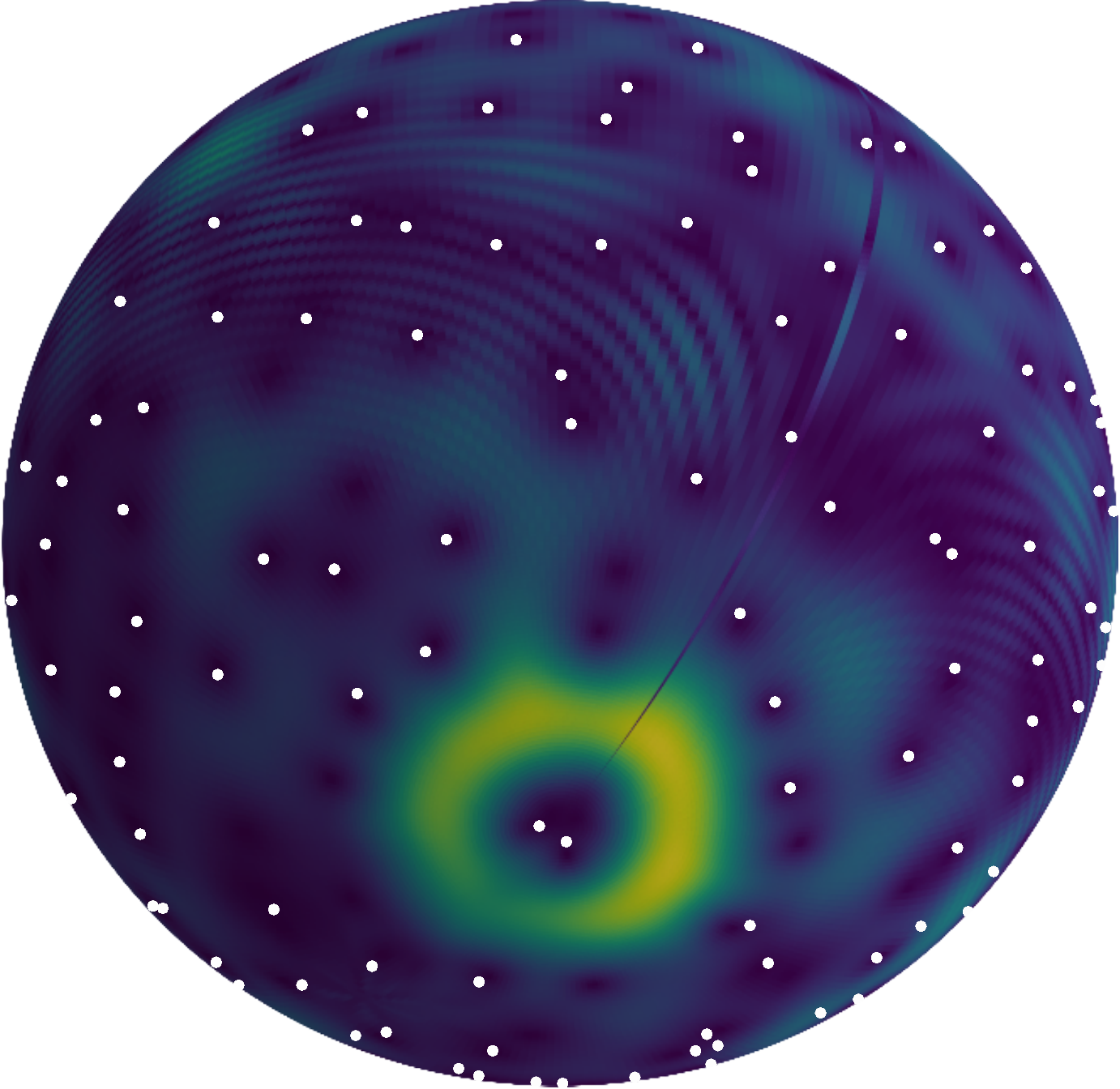}
  \label{f:sphericalDemo_detection}
  }
\subfigure[Inverse Fourier of Laguerre]{
  \includegraphics[width=\textwidth]{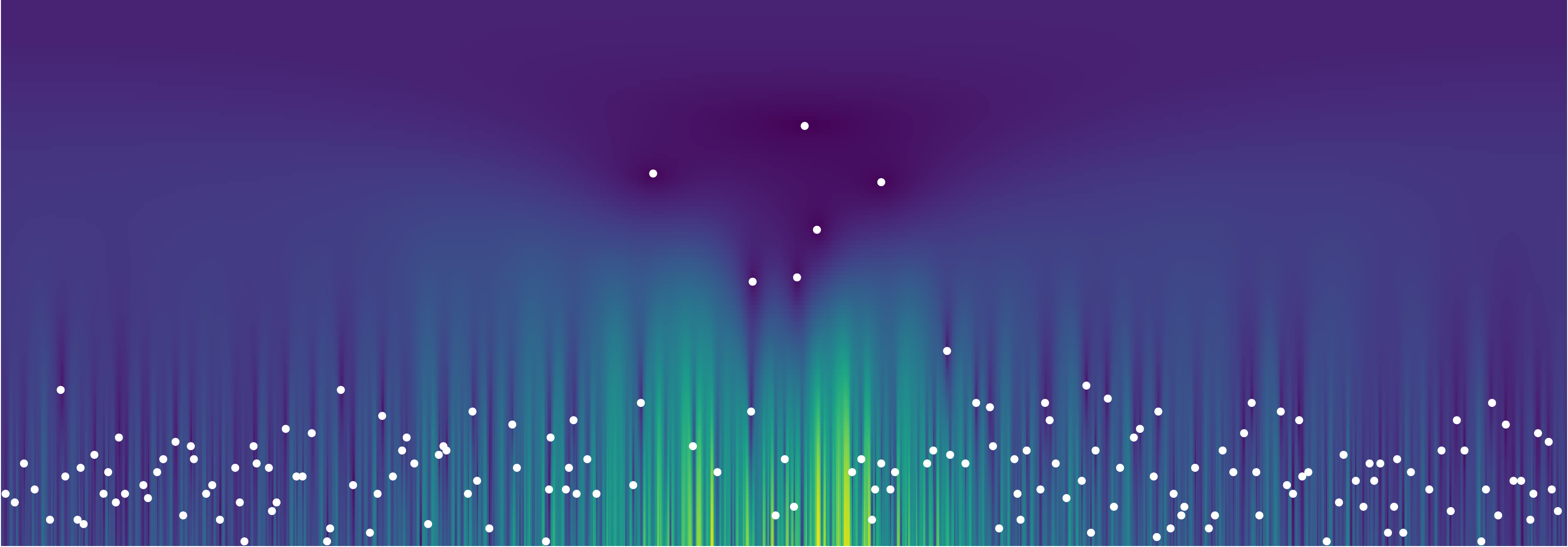}
  \label{f:hyperbolicDemo_detection}
  }
\caption{Example spectrograms/scalograms of signals containing one basis function $f_k$ hidden in white noise. The transforms used are the same as in Figure~\ref{f:gafs}. The noiseless signal is proportional to \ref{f:planarDemo_detection} a Charlier function, \ref{f:sphericalDemo_detection} a Krawtchouk function, and \ref{f:hyperbolicDemo_detection} the inverse Fourier transform of a Laguerre function.}
\label{f:gafs_detection}
\end{figure}

\subsubsection{The hyperbolic GAF}
\label{sec:hypGAF}
The \emph{hyperbolic GAF} of parameter $\alpha>-1$ is defined on $\Lambda=\mathbb D$ by
\eq
\label{hyperbolicGAF}
\GAF^{(\alpha)}_{\mathbb D}(z):=\sum_{k=0}^\infty \xi_k \,\sqrt{\frac{\Gamma(k+\alpha+1)}{k!}} z^k ,
\qe
where $(\xi_k)_{k\in\N}$ is a sequence of i.i.d. $N_\C(0,1)$ random variables. Its covariance kernel is given by
$$
\E\big[\GAF^{(\alpha)}_{\mathbb D}(z)\overline{\GAF^{(\alpha)}_{\mathbb D}(z)}\big]= \frac1{(1-z\overline{w})^{\alpha+1}}.
$$
where, when $\alpha$ is not an integer, we pick the branch of the logarithm so that $(1-|z|^2)^{\alpha+1}>0$.
The zeros of \eqref{hyperbolicGAF} are invariant under the isometries of $(\mathbb D, (1-|z|^2)^{-2}g_\mathrm{flat})$, which consist in the homographies $T:\mathbb D\to\mathbb D$,
\eq
\label{isomD}
T(z)= \frac{a z+b}{\overline b z+\overline a},\qquad  a,b\in\C,\qquad |a|^2-|b|^2=1.
\qe
\cite{PeVi05} showed that the zeros of the hyperbolic GAF of parameter $\alpha=0$ further have a special structure: they form a \emph{determinantal point process} (DPP) with Lebesgue reference measure and kernel the \emph{Bergman kernel} of $\mathbb D$,
\eq
\label{e:bergmanKernel}
K_\D(z,w):=\frac{1}{\pi(1-z\overline w)^2}.
\qe
In short, this means that the joint intensities \eqref{e:correlationIntensities} of the zeros of $\GAF_\mathbb{D}^{(0)}$ are explicit for any $k\geq 1$ and read
\eq
\d\rho_k(z_1,\ldots,z_k)=\det\Big[K_\D(z_i,z_j)\Big]_{i,j=1}^k\prod_{i=1}^k\bs 1_{\mathbb D}(z_k)\d^2z_k
\label{e:bergmanDPP}
\qe
where $\d^2z$ stands for the Lebesgue measure on $\C$. For more about determinantal point processes, we refer the reader to \citep{HKPV06,LaMoRu14}. As a side note, we add that $K_\D$ is the reproducing kernel of the Bergman space $A^2(\D):=\mathcal{A}(\D)\cap L^2(\d^2 z)$.

\paragraph{Daubechies-Paul wavelets and Laguerre polynomials.}
In our next Theorem~\ref{t:Hyp}, we link the hyperbolic GAF to a classical wavelet transform on analytic signals, known as the Cauchy or Daubechies-Paul wavelet transform after \citep{DaPa88}. More precisely, let $H^2(\mathbb{R})$ be the Hilbert space of analytic signals
$$
H^2(\R) := \{f\in L^2(\mathbb{R},\mathbb{C}) :\; \text{Supp}(\hat{f}) \subset \mathbb{R}_+ \},
$$
equipped with usual $L^2$ inner product $\langle f,g\rangle=\int \overline f g\,\d x$.  We refer to e.g. \cite[Theorem 5.9.5]{Sim153} for basic properties of $H^2(\R)$. We defined here the Fourier transform as $\hat{f}(\omega) = (2\pi)^{-1/2} \int \overline {f}(t)\e^{-\i\omega t}\d t$. 
 Now, let $\beta>-1/2$ and set
$$ \psi_\beta(t) := \frac{1}{(t+\i)^\beta},\quad t\in\mathbb{R}.$$
The Daubechies-Paul wavelet transform of $f\in H^2(\R)$ is defined  as
\eq
W_\beta f(u,s) := \langle f, T_u D_s \psi_\beta\rangle,
\label{e:pdwt}
\qe
where $u\in\mathbb{R}$ and $s\in\mathbb{R}_+^*$ are thought of as time and scale, $T_u f (t) := f(t-u)$ and $D_s f(t):= s^{-1/2} f(t/s)$ are the usual translation and dilation operators. The role of the Bargmann transform for the STFT is played here by the so-called Bergman transform of the unit disk $\D$, as investigated by \cite{Abr08}. In short, passing to the Fourier domain and using the identity
\eq
\label{inFourierPsiBeta}
 \frac{1}{\sqrt{2\pi}} \int_{\mathbb{R}_+} x^\beta \e^{-x}\e^{\i tx} \d x = -\frac {1}{\sqrt{2\pi}} \frac{(-\i)^{\beta+1}}{(t+\i)^{\beta+1}}\Gamma(\beta+1),
 \qe
the transform in \eqref{e:pdwt} can be rewritten as
\eq
\label{e:DPtransform}
W_\beta f(-u,s)= -\frac{\i^{\beta+1}s^{\beta+1/2}\sqrt{2\pi}}{\Gamma(\beta+1)} \left(\frac12-\i (u+\i s)\right)^{-2\beta-1} \L_\D^{(\beta)}f(\phi(u+\i s)),
\qe
where $u+\i s\in\mathbb{C}_+$  and we defined
\eq
\L_\D^{(\beta)}f(z) := \frac1{(1-z)^{2\beta+1}} \int_{\R_+} \overline{\hat f(x)}\, x^{\beta}\e^{-\frac x 2\tfrac{1+z}{1-z}} \d x, \quad z\in\D,
\label{e:bergmanTransform}
\qe
as well as $\phi(w) := (2w-\i)/(2w+\i)$, which is a conformal mapping from $\mathbb{C}_+$ to $\mathbb{D}$.
We note that writing $z=\phi(w)$, our transform \eqref{e:bergmanTransform} becomes
$$\L_\D^{(\beta)}f(\phi(w)) = \left(\frac12-\i w\right)^{2\beta+1}\int _{\mathbb{R}_+} \overline{\hat f(x)} x^{\beta}\e^{\i w x}\d x,$$
and we recover the Bergman transform of the disk, as in \citep{Abr08}, up to a nonvanishing factor that we add for later mathematical convenience.

\begin{theorem} Let $\beta>-1/2$ and $(f_k)$ be the basis of $H^2(\mathbb{R})$ consisting of the inverse Fourier transforms of the Laguerre functions with parameter $2\beta$, see \eqref{LagONB}. Let $\xi$ be the white noise on $H^2(\mathbb{R})$ associated with the basis $(f_k)$, as defined in Proposition~\ref{p:white}. Then
$$
\L_\D^{(\beta)}\xi \law \GAF_{\mathbb D}^{(2\beta)},
$$
and
$$
\L_\D^{(\beta)}f_k(z) = \sqrt{\frac{\Gamma(k+2\beta+1)}{k!}}z^k.
$$
In particular, the zeros of the Daubechies-Paul wavelet transform of white noise $z=u+\i s\mapsto W_\beta\xi(-u,s)$ are invariant under the isometries of the Poincar\'{e} half-plane $\C_+$. Moreover, when $\beta=0$, the zeros form a determinantal point process  on $\C_+$ with Lebesgue reference measure and kernel the Bergman kernel of $\C_+$,
\eq
\label{e:bergmanKernelC+}
K_{\C_+}(z,w):=\frac{-1}{\pi(z-\overline w)^2}\,,\qquad z,w\in\C_+.
\qe
\label{t:Hyp}
\end{theorem}

The proof of Theorem~\ref{t:Hyp} is provided in Section~\ref{s:analyticWavelets}.

A sample of the zeros of the hyperbolic GAF is shown in Figure~\ref{f:hyperbolicDemo}, which we obtained by truncating and discretizing the Daubechies-Paul wavelet transform of a truncated analytic white noise. We discuss the effect of truncation in Section~\ref{s:computation}, see also \cite[Section 5.1]{BaFlCh18}. The first few functions $(f_k)$ of our basis of $H^2(\mathbb{R})$ are plotted in Figure~\ref{f:inverseLaguerre}. The effect of a single $f_k$ with a real coefficient being added to an analytic white noise is shown in Figure~\ref{f:hyperbolicDemo_detection}. Note how a few zeros are trapped in a hyperbolic triangle, similarly to the zeros trapped inside the circular support of the signal in Figure~\ref{f:planarDemo_detection}. We investigate this trapping effect more formally in Section~\ref{s:finiteApproximationTransform}.

While we were preparing this manuscript, a similar result to Theorem~\ref{t:Hyp} was independently arXived by \cite*{AHKR18Sub}, with a different definition of white noise. Indeed, the white noise used in their paper is defined as a Gaussian space indexed by $L^2$ functions. In contrast, our white noise is defined as a  random variable living on functions that are slightly less regular than those in $H^2(\R)$, see Section~\ref{Supp2}. Our white noise may thus come closer to what signal processers call an analytic white noise. This answers a question raised in \citep[Section 4.4]{BaFlCh18}, see also the original construction of the analytic white noise by \cite{Pug82}. Moreover, our construction of white noise yields a precise definition of its transform as a natural extension, see Section~\ref{s:white}.

\paragraph{A discrete STFT and Meixner polynomials.}

Similarly to what the windowed discrete Fourier transform \eqref{e:cmp} is to the STFT, we also found a discrete counterpart to Theorem~\ref{t:Hyp}. For $f\in\ell^2(\mathbb{N},\mathbb{C})$, $z\in\mathbb{D}$ and $\beta>-1/2$, let
\beq
\ell_\mathbb{D}^{(\alpha)} f(z)  := \sum_{x\in\N} \overline{f(x)} \sqrt{\frac{\Gamma(x+\alpha+1)}{x!}}\, z^x\,.
\eeq
From a time-frequency perspective, we let $z=r\e^{\i\theta}$ with $r<1$, and obtain another windowed discrete Fourier transform
\beq
\ell_\mathbb{D}^{(\alpha)} f(z)  = \sum_{x\in\N} \overline{f(x)} \sqrt{\frac{\Gamma(x+\alpha+1)}{x!}}\, r^x \e^{\i x\theta}\,,
\eeq
again computationally amenable to FFT. The window here can be interpreted as a Conway-Maxwell-negative binomial distribution. Note that when $\alpha=0$, the window is proportional to a simple geometric distribution with mean $1/r$.

As in the discussion that followed \eqref{e:charlierTransform}, $\ell_\mathbb{D}^{(\alpha)}$ maps the canonical basis of $\ell^2(\mathbb{\N})$ onto monomials. But as in Theorem~\ref{t:planarl},  we show in Theorem~\ref{t:hyp} that it is possible to rather map discrete orthogonal polynomials onto a simple class of polynomials.

More precisely, for $\alpha>-1$ and $0<c<1$. Let $(f_k)$ be the Meixner functions coming with the Meixner polynomials. The Meixner polynomials and functions are defined later in \eqref{e:meixnerFunction}, but for now we simply note that $(f_k)$ is a basis of $\ell^2(\mathbb{N},\mathbb{C})$ and that $f_k$ oscillates faster with $k$, similarly to Charlier functions in Figure~\ref{f:charlier}.

\begin{theorem} Let $\xi$ be a white noise on $\ell^2(\N,\C)$ associated with the Meixner functions $(f_k)$ of parameters $\alpha>-1$ and $0<c<1$, see Proposition~\ref{p:white}. We have,
\eq
\ell_\mathbb{D}^{(\alpha)}\xi \law  \GAF_{\mathbb D}^{(\alpha)}
\label{e:hypStatement}
\qe
and
$$
\ell_\mathbb{D}^{(\alpha)}f_k (z)=\sqrt{\frac{\Gamma(k+\alpha+1)}{k!}}(\sqrt c-z)^k(1-\sqrt cz)^{\alpha+1-k}.
$$
In particular, the zeros of $\ell_\D^{(0)} \xi$ form a determinantal point process on $\D$ with Lebesgue reference measure and kernel  $K_{\D}$ defined in \eqref{e:bergmanKernel}.
\label{t:hyp}
\end{theorem}
We prove Theorem~\ref{t:hyp} in Section~\ref{sec:proofMeixner}.

\subsubsection{The spherical GAF}
Finally, the \emph{spherical GAF} of parameter $N\in\{1,2,3,\ldots\}$ is defined on $\Lambda=\C$ by
\eq
\label{sphericalGAF}
\GAF^{(N)}_{\mathbb S}(z):=\sum_{k=0}^N \xi_k \,\sqrt{{N}\choose{k}} \,z^k ,
\qe
where $(\xi_k)_{k\in\N}$ is a sequence of i.i.d. $N_\C(0,1)$ random variables. The covariance structure reads
$$
\E\Big[\GAF^{(N)}_{\mathbb S}(z)\overline{\GAF^{(N)}_{\mathbb S}(z)} \,\Big]=(1-z\overline w)^N.
$$
Its zeros are invariant under the isometries of $(\mathbb C\cup\{\infty\}, (1+|z|^2)^{-2}g_\mathrm{flat})$, which are the maps $T:\C\cup\{\infty\}\to\C\cup\{\infty\}$,
\eq
\label{e:isomSphere}
T(z)=\frac{az+b}{-\overline{b}z+\overline{a}}, \qquad  a,b\in\C,\qquad |a|^2+|b|^2=1.
\qe

In the same spirit as the previous two discrete STFTs introduced for the planar and hyperbolic GAFs, we consider now the windowed discrete Fourier transform acting on $\C^{N+1}$ as
$$
\ell_{\mathbb S}^{(N)}f(z) := \sum_{x=0}^N \overline{f(x)} \sqrt{ {N \choose x}} \,z^x
$$
where $f=(f(x))_{x=0,\dots,N}\in \mathbb{C}^{N+1}$. By writing  $z=r\e^{\i \theta}\in\mathbb{C}$, the window here can be seen as proportional to a Conway-Maxwell-binomial distribution with parameters $\nu=1/2$ and $p=r/(1+r)$. This distribution is akin to an overdispersed binomial distribution, and it is unimodal with mode $(N+1)\lfloor r^2/(1+r^2)\rfloor$.

Clearly $\ell_{\mathbb S}^{(N)}$ maps the canonical basis to monomials and, if $\xi$ stands for the standard Gaussian variable on $\C^{N+1}$, which is defined unambiguously, then $\ell_{\mathbb S}^{(N)}\xi(z)\law \GAF_{\mathbb S}^{(N)}(z).$  Again, the next result shows the same holds true when replacing the canonical basis by the Krawtchouk functions. The Krawtchouk functions thus play here the role of the Hermite functions for the STFT with Gaussian window. Indeed, let $p\in(0,1)$ and define the Krawtchouk functions \eqref{e:fkKrawchouk} associated with the Krawtchouk polynomials --~the discrete orthogonal polynomials associated with the Binomial distribution $\mathcal B(N,p)$. To gain intuition on these  elementary waveforms, we plot the first Krawtchouk functions in Figure~\ref{f:kravchuk}.

\begin{theorem}
\label{t:spherical} Let $N\in\N$ and $\xi$ be the usual white noise on $\C^{N+1}$, that is a standard complex Gaussian random vector of dimension $N+1$.  We have,$$
\ell_{\mathbb S}^{(N)}\xi\law\GAF_{\mathbb S}^{(N)}\,.
$$
Moreover, if $(f_k)$ is the basis of Krawtchouk functions of parameter $N$ and $p\in(0,1)$, then
$$
\ell_{\mathbb S}^{(N)}f_k(z) = \sqrt{\frac1{(1+q^2)^N} {N \choose k}} (1-qz)^k(q+z)^{N-k},\qquad q:=\sqrt{\frac{1-p}{p}}.
$$
\end{theorem}

The proof of Theorem \ref{t:spherical} is given in Section \ref{sec:proofDiscF}.

We illustrate a sample of the zeros of $\ell_{\mathbb S}^{(N)}\xi$ in Figure~\ref{f:sphericalDemo}, with a change of variables $z\mapsto (1-qz)/(q+z)$ to map Krawtchouk functions to monomials. We stress that this change of variables leaves the spherical GAF invariant, see the proof in Section~\ref{sec:proofDiscF}. This is rather curious, but it underlines the duality of the canonical basis and the Krawtchouk basis.
Note in Figure~\ref{f:sphericalDemo} how the zeros exhibit the same super-uniformity as in Figure~\ref{f:planarDemo} but on the sphere. The spectrogram of a single Krawtchouk function with additional complex white noise is shown in Figure~\ref{f:hyperbolicDemo_detection}. As expected, the support of the signal is again ring-shaped  with a few zeros trapped inside the inner spherical cap of the ring, similarly to what happens for Hermite functions and the Gabor transform, or Charlier functions and the STDFT of Section~\ref{s:planar}.

\subsection{More transforms, more GAFs, and related results}

The previous results turn out to be particular cases of a general setting that we present later in Section~\ref{sec:white}. As we shall see, each of these previous results involves an identity for generating functions of classical orthogonal polynomials. By using the same approach, it is possible to obtain other results involving other GAFs than the planar, hyperbolic, or spherical GAFs. We collect some of these results in this section. First, we briefly mention real white noises as opposed to complex white noises. Then we  link the hyperbolic GAF with the analytic projection of the white noise on the unit circle. We also relate the derivatives of the Bargmann transform to a one-parameter deformation of the planar GAF and discuss what our results can say on the local extrema of GAFs. Since time-frequency filtering algorithms have historically been based on spectrogram maxima rather than zeros, this may prove of interest in the future.

\paragraph{Real white noises and symmetric GAFs.}
In this paper, our white noises are defined on Hilbert spaces that are vector spaces on $\mathbb{C}$. The whole constructions of the transform and the white noise in Section~\ref{s:white} carry straightforwardly to \emph{real} Hilbert spaces. The resulting analytic functions are termed \emph{symmetric GAFs} \citep{Fel13}. Symmetric GAFs may be of less computational interest to signal processers, since their zeros have less invariance than the zeros of their complex counterparts. They are thus less amenable to traditional statistical signal processing procedures, which often assume some stationarity. However, zeros of symmetric GAFs are usually well approximated by the zeros of GAFs, see \citep{Pro96,Fel13}. We refer the reader to \citep[Section 4.3]{BaFlCh18} for another discussion on the Gabor transform of a real white noise and its approximation by that of a proper complex white noise.

\paragraph{Analytic projections and trigonometric polynomials.}
Another model for the hyperbolic GAF, already discovered in \cite[Section 5.4.3]{{HKPV09}}, is the analytic projection of the white noise on the unit circle.  We now rephrase this result in a slightly more general version. More precisely, consider for any $N\in\N$ the transform
$$
\L^{\mathsf{per}}_{\D,N} f(z):=\frac1{\sqrt{2\pi}}\int_0^{2\pi} \overline{f(\theta)} \,\frac{\e^{\i N \theta}}{1-z\e^{\i\theta}}\, \d \theta,\qquad z\in\D,
$$
acting on any $f\in L^2([0,2\pi],\C)$.

\begin{theorem}
\label{th:analyticproj} Let $\xi$ be the white noise on $L^2([0,2\pi],\C)$ associated with the Fourier basis $(f_k)$ where $k\in\mathbb Z$. Then, for any $N\in\N$, we have
$$
\L^{\mathsf{per}}_{\D,N}f_k(z)=
\begin{cases}
z^{k-N}& \text{ if } k\geq N,\\
0 & \text{ if } k<N.
\end{cases}
$$
Moreover,
$$
\L^{\mathsf{per}}_{\D,N}\xi\law \GAF_\D^{(0)}.
$$
In particular the zeros of $\L^{\mathsf{per}}_{\D,N}\xi$ form a determinantal point process on $\D$ with Lebesgue reference measure and Bergman kernel $K_\D(z,w)$ defined in \eqref{e:bergmanKernel}.

\end{theorem}

The proof of Theorem~\ref{th:analyticproj}  is given in Section~\ref{proof:analyticproj}. We remark that the parameter $N$ does not impact the resulting GAF. Finally, a discussion on where the white noise lives is provided in Section~\ref{Supp1}. Finally,

\paragraph{Derivatives of the Bargmann transform and application to local extrema.}
We consider here the higher order Bargmann transform defined, for $N\in\N$ and $f\in L^2(\R,\C)$, by
\eq
\label{bargN}
\L_{\C,N} f(z):=\frac{\e^{-z^2/2}}{\pi^{1/4}} \int_\R \overline{f(x)}\,H_ N(x-z)\,\e^{\sqrt2 xz-x^2/2}\d x,\qquad z\in\C\,,
\qe
where $H_N$ is the $N$-th degree Hermite polynomial,
$$
H_N(x):=(-1)^N\e^{x^2}\frac{\d^N}{\d x^N}\,\e^{-x^2}=\sum_{k=0}^{\lfloor  N/2\rfloor}\frac{N!(-1)^k(2x)^{N-2k}}{k!(N-2k)!}.
$$
We immediately recover the usual Bargmann transform \eqref{barg} when taking $N=0$. Moreover, let us introduce the new GAF
$$
\GAF_{\C,N}^{(1)}(z):=\sum_{k\geq N} \xi_k \frac{\sqrt{k!}}{(k-N)!}\, z^k ,
$$
where $(\xi_k)$ is a sequence of i.i.d. $N_\C(0,1)$ random variables. Note that we recover the planar GAF when $N=0$. It has for covariance kernel
\begin{align}
\label{covNew}
\E\Big[\GAF_{\C,N}^{(1)}(z)\overline{\GAF_{\C,N}^{(1)}(w)}\,\Big] &
 =N!(z\overline{w})^N\sum_{k=0}^\infty {N+k\choose k} \frac{(z\overline{w})^k}{k!} \nonumber\\
& = N!(z\overline{w})^N  \e^{z\overline{w}} \sum_{\ell=0}^N  {N\choose \ell}\frac{(z\overline{w})^\ell}{\ell!}\nonumber\\
& = N!(z\overline{w})^N  \e^{z\overline{w}} L^{(0)}_N(-z\overline{w}),
\end{align}
where $L^{(0)}_N$ is the $N$-th degree Laguerre polynomial of parameter $\alpha=0$, see also \eqref{Lag}--\eqref{Lag2}. For the second equality, we used that
$$
\sum_{\ell=0}^N{k\choose \ell}{N\choose \ell}={N+k\choose k}\, ,
$$
which is a particular case of the  Chu-Vandermonde identity.

  The next theorem is a generalization of Theorem~\ref{t:planarL}, which is the special case $N=0$; recall the Bargmann transform $\L_\C$ was introduced in \eqref{barg}.

\begin{theorem}
\label{t:planarLN}
Let $\xi$ be the white noise on $L^2(\R,\C)$ associated with the Hermite functions $(f_k)$, as defined in Proposition \ref{p:white}. Then, we have for any $N\in\N$,
$$
\L_{\mathrm{\C,N}}\xi\law \GAF_{\C,N}^{(1)}.
$$
Moreover, for any $N\in\N$ and $f\in L^2(\R,\C)$,
$$
\frac{\d^N}{\d z^N} \L_\C f(z)= \L_{\C,N} f(z),
$$
and in particular
\eq
\label{bargNmono}
\L_{\C,N}f_k (z)= \begin{cases}
 \displaystyle\frac{\sqrt{k!}}{(k-N)!}\, z^{k-N} &\text{ if } k\geq N\\
0 &\text{ if } k< N
\end{cases}.
\qe
\end{theorem}
The proof of Theorem~\ref{t:planarLN} is given in Section~\ref{e:gabor}. As a consequence of \eqref{bargNmono}, we see that $\L_{\C,N}$ maps any signal $f\in\mathrm{Span}_\C(f_0,\ldots,f_{N-1})$ to zero. The transform thus acts as a high-pass filter, and can be used to isolate higher orders harmonics of a signal.

Although this might change with more knowledge on the zeros becoming available, classical time-frequency algorithms rely on the maxima of spectrograms rather than their zeros. As pointed out in Section~\ref{s:GAFsAndTheirZeros}, any nonzero GAF has simple random zeros with probability one. Its derivative also being a GAF, the random zeros of the derivative are necessary local extrema of the original GAF, not saddle points. Moreover, they are completely characterized by their first intensity. By taking $N=1$ in the previous theorem we obtain the following result on the local extrema of the Bargmann transform of the white noise. Note, however, that these extrema are not the ones of the corresponding STFT since the transforms differ by a factor $\e^{-|z|^2}$ \cite[Proposition 3.4.1]{Gro01}.

\begin{corollary} Let $\xi$ be the white noise as in Theorem~\ref{t:planarLN}.  The local extrema of the Bargmann transform of the white noise $\L_\C\xi$ have the same distribution as the zeros of $\GAF^{(1)}_{\C,1}$. In particular, they are  characterized by their joint intensity given by
$$
\d\rho_1(z)=\frac1\pi\left(1+ \frac1{(1+|z|^2)^2}\right)\d^2z.
$$

\end{corollary}

\begin{proof} According to \eqref{covNew} and Theorem~\ref{t:planarLN} with $N=1$, the local extrema of $\L_\C\xi$ are exactly the zeros of $\GAF_{\C,1}^{(1)}$n with covariance structure
$$
\kappa(z,w):=z\overline{w}(1+z\overline{w})\e^{z\overline w},\qquad z,w\in\C.
$$
The first intensity follows from the Edelman-Kostlan formula, namely
$$
\d\rho_1(z)=\frac1{4\pi}\Delta\log\kappa(z,z)
$$
in the distributional sense, see e.g. \cite[Equation (2.4.8)]{HKPV09}.
\end{proof}

\section{White noises and their transforms}
\label{s:white}

\label{sec:white}
In this section we introduce a class of transformations $\L$ acting on Hilbert spaces denoted by $\H$ and the white noise of this space. $\H$ should be thought of as a space of signals $f$, and $\L f(z)$ as a time-frequency or time-scale transform. The way the complex variable $z$ is related to time and frequency/scale will vary across the section, sometimes simply matching time to the real part and frequency/scale to the imaginary part, sometimes rather using the polar decomposition. Since our goal is to understand the zeros of transforms of white noise, we also specify in this section what we mean by a white noise $\xi$ on $\H$, and how to define and understand its transform $\L \xi$. We conclude this section by a summary of the results we gathered; the reader may jump directly to Proposition~\ref{p:white} for the definition of the white noise used in this work and to Theorem~\ref{t:main} for the template theorem from which all the previous theorems will follow once combined with appropriate generating function identities.

\subsection{The operator $\L$}
\label{sec:Ltrans}
Let $\H$ be a complex Hilbert space with inner product $\langle \,\cdot\,,\,\cdot\,\rangle$ and associated norm $\|\cdot\|$. We always assume that $\H$ is separable, or equivalently that there exists a countable orthonormal basis for $\H$. Fix an open subset $\Lambda\subset\C$ of the complex plane. Important examples of $\Lambda$ in this paper are the complex plane $\C$, the unit disc $\mathbb D:=\{z\in\C:\;|z|<1\}$, or the upper half-plane $\C_+:=\{z\in\C:\;\Im(z)>0 \}$. 
Given an orthonormal basis $(f_k)_{k\in\N}$ of $\H$ and a sequence $(\Psi_k)_{k\in\N}$ of holomorphic functions on $\Lambda$ satisfying
\eq
\label{Ass}
\sup_{z\in K}\,\sum_{k\in\N} |\Psi_k(z)|^2<\infty\,
\qe
for any compact $K\subset \Lambda$, we define the operator $\L$ acting on $\H$ as
\eq
\label{L}
\L f(z):=\sum_{k=0}^\infty \langle f_k,f\rangle \Psi_k(z),\qquad z\in\Lambda.
\qe
In signal processing terms, the signal $f\in\H$ is analyzed in the basis $(f_k)$, and reconstructed as an analytic function in $\mathcal A(\Lambda)$ using the dictionary $\Psi_k$. Note that, by dominated convergence, we can alternately write the transform $\L$ in kernel form, namely for any $z\in\Lambda$,
\eq
\label{Lbis}
\L f(z):=\langle T(\cdot,z),f\rangle,\qquad T(x,z):= \sum_{k=0}^\infty  f_k(x) \Psi_k(z).
\qe
As we shall see in Section \ref{s:proofs}, the STFT with Gaussian window and the analytic wavelet transform of \cite{DaPa88} are examples of $\L$. In particular, the representation \eqref{Lbis}   is then provided by identities for the generating functions of classical orthogonal polynomials.

Meanwhile, we note that since for any $f\in\H$,
\eq
\label{ineq1}
|\L f(z)|\leq \sum_{k=0}^\infty |\langle f_k,f\rangle| \,|\Psi_k(z)|\leq \|f\| \sqrt{\sum_{k\in\N} |\Psi_k(z)|^2},
\qe
Assumption~\eqref{Ass} yields the absolute convergence of $\L f(z)$ and $T(\cdot,z)$, locally uniformly in $z\in\Lambda$. In particular, $\L f$ is a well-defined analytic function on $\Lambda$, being the uniform limit of analytic functions on every compact subset of $\Lambda$. Moreover, $\L$ is a continuous linear mapping from $\H$ to the space $\mathcal A(\Lambda)$ of analytic functions on $\Lambda$, equipped with the topology of the local uniform convergence.

\subsection{White noises and the abstract Wiener spaces of \cite{Gro67}}

A white noise on an arbitrary normed space is intuitively defined as a random variable $\xi$ with law the standard Gaussian measure on that space. The problem is that there is no ``standard Gaussian measure''  on infinite-dimensional spaces. In this subsection and for the sake of completeness, we survey the abstract Wiener space construction of \cite{Gro67}, a rigorous definition for a white noise on a general complex Hilbert space $\H$. We also argue that this definition fits the requirements of being a good model for signal processing.

\subsubsection{The finite-dimensional setting}
\label{s:finiteDimensionalWhiteNoise}
When $\H\simeq \C^{N}$, a white noise $\xi$ on $\H$ is simply defined by
\eq
\label{xifinitedim}
\xi:=\sum_{k=1}^N\xi_k\, f_k\, ,
\qe
where $(\xi_k)_{k=1}^N$ are i.i.d $N_\C(0,1)$ random variables and $(f_k)_{k=1}^N$ is any orthonormal basis of $\H$. The choice of the basis does not change the law of $\xi$, though in practice the canonical basis is often used for simulations in signal processing, where white noise is thought of as a sequence of $N$ i.i.d. Gaussian measurements. The law of $\xi$ is the standard Gaussian measure on $\H$, that is, the probability measure $\mu$ on $\H$ characterized by its (extended) Fourier transform:
\eq
\label{FTfinitedim}
\int_\H \e^{\i \langle f, h\rangle}\mu(\d h)= \E[\e^{\i \langle f, \xi\rangle}]=\e^{-\|f\|^2/2},\qquad f\in\H.
\qe
In particular $\mu$ is invariant under isometries of $\H$ and thus the definition \eqref{xifinitedim} of $\xi$ does not depend on the specific choice for the basis $(e_k)_{k=1}^N$. Moreover, it follows from \eqref{xifinitedim} that, for any $f\in\H$, we have $\langle f,\xi\rangle\law N_\C(0,\|f\|^2).$

\subsubsection{The infinite-dimensional setting}
\label{s:gross}
When $N=\infty$, we need to proceed differently since there is no (countably additive) probability measure $\mu$ supported on $\H$ satisfying \eqref{FTfinitedim} for every $\lambda\in\H$. In fact, although one can always construct a probability space on which an infinite sequence $(\xi_k)_{k\in\N}$ of i.i.d $N_\C(0,1)$ random variables lives, the series \eqref{xifinitedim} is almost surely divergent in $\H$; this is often paraphrased by saying that {$\cH$ is not big enough to accommodate a white noise $\xi$}. {Indeed, Parseval's identity and the law of large numbers  yield
$$
\left\|\sum_{k\in\N} \xi_kf_k\right\|=\sqrt{\sum_{k\in\N} |\xi_k|^2}=+\infty\quad\text{a.s.}
$$}
One way to solve this conceptual problem is to follow the approach of \cite{Gro67} and to embed $\H$ into a larger complex Banach space $\Theta$ that can support such a probability distribution. We refer the reader to \cite[Chapter 8]{Str10} for references and further information; note that although the general theory is stated there for real Hilbert spaces, it adapts to the complex setting without substantial modification.

The idea is to take as Banach space $\Theta$ the completion of $\H$ for a norm $\|\cdot\|_\Theta$ that is weaker than $\|\cdot\| = \|\cdot\|_\H$, making sure that the series \eqref{xifinitedim} converges almost surely in $\Theta$. Letting $\Theta^*$ be the space of continuous linear forms $\Theta\to\C$, we have the inclusion $\Theta^*\subset\H^*$ of topological duals. We can thus isometrically identify $\Theta^*$ as a dense subset of $\H$, using the mapping $\lambda \mapsto h_\lambda$, where $\lambda=\langle h_\lambda,\,\cdot\,\rangle$ is provided by the Riesz representation theorem. Now the question is whether there exists a probability measure $\mu$ supported on $\Theta$ satisfying
\eq
\label{FT}
\int_\Theta \e^{\i \lambda (\theta)}\mu(\d \theta)=\e^{-\|h_\lambda\|^2/2},\qquad \lambda\in\Theta^*.
\qe
One can show that \eqref{FT} characterizes the probability measure $\mu$ on $\Theta$ and that the left-hand side of \eqref{FT} is continuous in $\lambda$ for the weak* topology.

Assume for a moment that a measure $\mu$ satisfying \eqref{FT} exists. The triplet $(\H,\Theta,\mu)$ is then called an \emph{abstract Wiener space}. If $\xi$ is a random variable on $\Theta$ with law $\mu$, then it follows from \eqref{FT} that $\lambda(\xi)\law N_\C(0,\|h_\lambda\|^2)$ for any $\lambda\in\Theta^*$. Since $\Theta^*$ is dense in $\H^*\simeq\H$, the linear isometry $\lambda\in\Theta^*\mapsto \lambda(\xi)\in L^2(\mu,\C)$ extends uniquely to a linear isometry $\H\to L^2(\mu,\C)$. This provides a (rather abstract) meaning to $\langle f,\xi\rangle$ as an element of $L^2(\mu,\C)$ for any $f\in \H$, and we have $\langle f,\xi\rangle\law N_\C(0,\|f\|^2).$ Thus,
\eq
\label{eqFour}
 \E[\e^{\i \langle f, \xi\rangle}]=\e^{-\|f\|^2/2},\qquad f\in\H.
\qe
Having in mind \eqref{FTfinitedim}, it seems reasonable to say that $\xi$ is a white noise on $\H$.



There are however a few subtleties to keep in mind. First, still assuming that such a measure $\mu$ exists, it is supported on the whole of $\Theta$, and hence $\mu$ depends strongly on the choice of $\Theta$. Moreover, $\mu$ puts zero mass on $\H$: this may be puzzling, especially if one remembers that $\langle f,\xi\rangle$ is only defined $\mu$-almost everywhere when $\langle f,\cdot\rangle \in \H^*\setminus\Theta^*$. But recall that a measure $\mu$ fully supported on $\H$ satisfying \eqref{eqFour} for all $f\in\H$ does not exist, so one cannot hope for a better notion of white noise, except maybe for finding the ``smallest'' $\Theta$ possible.

Finally, we consider the existence of $\mu$. Following \cite{Gro67}, we say that the norm $\|\cdot\|_\Theta$ is \emph{measurable} if, for any $\epsilon>0$, there exists a finite dimensional subspace $E_\epsilon\subset \H$ such that, for any finite rank projection $P:\H\to\H$ with range perpendicular to $E_\epsilon$, we have
$
\P(\|P\xi\|_{\Theta}> \epsilon)<\epsilon.
$
Here, the notation $P\xi$ means a white noise on the (finite-dimensional) range of $P$, see Section~\ref{s:finiteDimensionalWhiteNoise}.
\begin{theorem}[\cite{Gro67}]
  \label{t:gross}
  A probability measure $\mu$ on $\Theta$ satisfying  \eqref{FT} exists if and only if  $\|\cdot\|_\Theta$ is measurable. Moreover, a sufficient condition for measurability is that $\|\cdot\|_\Theta=\|B\cdot\|$ with $B:\H\to \H$ a Hilbert-Schmidt operator.
\end{theorem}

\subsection{The white noise associated with $(f_k)$}
\label{sec:WN}
In this work, we define a white noise on a complex Hilbert space $\H$ by means of a distinguished orthonormal basis $(f_k)_{k\in\N}$, the one coming with the transform $\L$ defined in Section~\ref{sec:Ltrans}.  More precisely, we set
\eq
\label{Theta_norm}
\| f\|_\Theta^2:=\sum_{k\in\N}\frac1{1+k^2}\,\big|\langle f,f_k\rangle\big|^2
\qe
and define $\Theta$ as the completion of $\H$ with this norm. To see it is a measurable norm in the sense of Gross, define the operator $B$ on $\cH$ by setting $Bf_k:=(1+k^2)^{-1/2}f_k$ and extending $B$ to $\H$ by density. $B$ is a self-adjoint Hilbert-Schmidt operator satisfying $\| f\|_\Theta=\|Bf\|$, so that Theorem~\ref{t:gross} yields the existence of an abstract Wiener space $(\H,\Theta,\mu)$.

Next, we  define the \emph{white noise associated with $(f_k)$} as the random variable
\eq
\label{def:xi}
\xi:=\sum_{k\in\N}\xi_k \,f_k
\qe
which converges almost surely in $\Theta$. Indeed,
\eq
\label{Err}
\left\|\sum_{k\in\N} \xi_k \,f_k\right \|_\Theta^2= \sum_{k\in\N}\frac{|\xi_k|^2}{1+k^2}<\infty \quad \text{a.s.}
\qe
since the right hand side has finite expectation. Moreover $\xi$ has for distribution $\mu$, as one can see by truncating the series \eqref{def:xi} up to the $N$ first terms, computing its Fourier  transform, and letting $N\to\infty$ so as to obtain \eqref{FT}.

 \subsection{Support of the white noise}
 Instead of using the abstract Wiener spaces, one could have constructed the white noise as a random variable taking its values in the space $S'(\mathbb{T},\mathbb{C})$ of tempered distributions, at least when $\H$ is some $L^2$ space of functions, see e.g.  \citep[Section 2.1]{HOUZ10}. This is the approach followed by \cite{BaFlCh18} to identify the Gabor transform of white noise.  However, the search for ``smaller'' abstract Wiener spaces like $\Theta$ is philosophically interesting: since elements outside of $\H$ are usually hard to interpret physically, it is desirable to add as few elements as possible to the space $\H$. Let us now describe the support $\Theta$ of the white noise used in this work in several examples and observe that it is smaller than $S'(\mathbb{T},\mathbb{C})$. To do so we use the differential equations satisfied by the classical orthogonal polynomials.

First, let us consider the following subspace of $\H$,
$$
\mathsf H:=\left\{f\in\H:\;\sum_{k\in\N}(1+k^2)\big|\langle f,f_k\rangle\big|^2<\infty \right\}.
$$
The Cauchy-Schwarz inequality then yields the useful criterion: for any $f\in \H$,
\eq
\label{inTheta*}
f\in\mathsf H \qquad \Rightarrow\qquad \langle f,\,\cdot\,\rangle\in \Theta^*.
\qe
Typically \eqref{inTheta*} is an equivalence, but this inclusion is enough for our purpose. Next, since the norm $\|\cdot\|_\Theta$ arise from the inner product
$$
\langle f,g\rangle_\Theta:=\sum_{k\in\N}\frac{1}{1+k^2}\langle f,f_k\rangle\overline{\langle g,f_k\rangle}
$$
making $\Theta$ an Hilbert space, then $\Theta$ is reflexive and thus we have the inclusion
$$
\Theta\subset\mathsf H^*.
$$
{Since one can typically provide a description of the space $\mathsf H^*$ as the dual of a Sobolev-type space in practice, this yields ``an upper bound'' on the regularity of function space  $\Theta$ where the white noise takes its values. }
 
Let us now provide a few examples where one can identify the space $\mathsf H^*$.

\subsubsection{Example 1: White noise on the unit circle}
\label{Supp1}
The simplest setting to describe the white noise is the space $\H:=L^2([0,2\pi],\C)$ appearing in Theorem~\ref{th:analyticproj}. Indeed,
equipp $\H$ with inner product $\langle f,g\rangle:=\int_0^{2\pi} \overline{ f } g\,\d\theta$ and Fourier basis $f_k(\theta):=\e^{\i k\theta}/\sqrt{2\pi}$ with $k\in\Z$. Since $\langle f,f_k\rangle$ are the Fourier coefficients of $f$, we have
\eq
\label{H1norm}
\sum_{k\in\Z}(1+k^2)\big|\langle f,f_k\rangle\big|^2=\sum_{k\in\Z}\big|\langle f,f_k\rangle\big|^2+\big|k\langle f,f_k\rangle\big|^2=\| f\|^2+\| f'\|^2
\qe
and $\mathsf H$ is known as the Sobolev space $H^1([0,2\pi],\C)$. Its dual $\mathsf H^*$ is known as the dual Sobolev space  $H^{-1}([0,2\pi],\C)$ and thus $\Theta\subset H^{-1}([0,2\pi],\C)$ (in fact we have equality here). This space can alternatively be described as the space of complex-valued tempered distributions $\sum_{k\in\Z}a_k f_k$ on $[0,2\pi]$ with Fourier coefficients satisfying
\eq
\sum_{k\in\Z}\frac{|a_k|^2}{1+k^2}<\infty.
\label{e:conditionH1Prime}
\qe
Note that such distributions cannot behave too wildly: for instance, they contain distributional derivatives of functions in $\H$, but taking the second derivative may already violate \eqref{e:conditionH1Prime}. This shows that $\Theta$ is strictly contained in the space $S'(\mathbb{T},\mathbb{C})$ of tempered distributions.

\subsubsection{Example 2: White noise on analytic signals}
\label{Supp2}

We now describe the white noise of Theorem~\ref{t:Hyp} on the space $\H=H^2(\R)$ of analytic signals, see Section~\ref{sec:hypGAF}, equipped with the basis $(f_k)$ given by the inverse Fourier transform of the Laguerre functions of parameter $2\beta$ for $\beta>-1/2$, see~\eqref{LagONB}. To do so, let $\mathcal D$ be the differential operator formally defined on $L^2(\R_+,\C)$ by
$$
\mathcal D f(x) := -x\frac{\d^2}{\d x^2} f (x)- \frac{\d}{\d x} f (x)+\frac14\left(x+\frac{4\beta^2}x-2(1+2\beta)\right)f(x).
$$
The differential equation satisfied by Laguerre polynomials yields
$\mathcal D \,\hat f_k=k\, \hat f_k$
for every $k\in\N$, see \cite[Equation (4.6.16)]{Ism05}. Since $(\hat f_k)$ is a basis of $L^2(\R_+,\C)$ this yields the existence of a self-adjoint realization of $\mathcal D$, that we still denote $\mathcal D$, defined on the whole of $L^2(\R_+,\C)$. Thus,  we have
\begin{align*}
\sum_{k\in\Z}(1+k^2)\big|\langle f,f_k\rangle\big|^2& =\|f\|^2+\sum_{k\in\Z}\big|\langle \hat f, k\hat f_k\rangle\big|^2\\
& = \|f\|^2+\sum_{k\in\Z}\big|\langle \hat f, \mathcal D\hat f_k\rangle\big|^2\\
&=\|f\|^2+\sum_{k\in\Z}\big|\langle  \mathcal D\hat f,\hat f_k\rangle\big|^2\\
& = {\|\hat f\|^2}_{L^2(\R_+,\C)}+{\| \mathcal D\hat f \|^2}_{L^2(\R_+,\C)}\,,
\end{align*}
which should be compared to \eqref{H1norm}. This yields $$
\mathsf H=\Big\{f\in L^2(\R,\C) :\; \mathrm{Supp}(\hat f)\subset \R_+ ,\; \mathcal D\hat f\in L^2(\R_+,\C) \Big\}.
$$
Note that the condition  $\mathcal D\hat f\in L^2(\R_+,\C)$ is similar to $\hat f$ being member of some Sobolev space. Thus $\Theta$ is contained in the space of distributions on $\R$ whose Fourier transform lives in the dual space of $\{g\in L^2(\R_+,\C) :\;  \mathcal D g\in L^2(\R_+,\C)\}$.

{
\begin{remark}Note that a similar analysis can be done for the white noise on $L^2(\R)$ associated with the Hermite functions $(f_k)$, since $f_k$ is also the eigenfunction associated with the eigenvalue $k$ of the second order differential operator:
$$
\mathcal D f(x):=-\frac{\d^2}{\d x^2} f (x) +\left(\frac{x^2-2}{4}\right)f(x).
$$
Thus, we have in this setting
$$
\mathsf H =\Big\{f\in L^2(\R,\C) :\;  \mathcal D  f\in L^2(\R,\C) \Big\}
$$
and we see that $\mathsf H$ contains the $L^2$ functions $f$ such that  $f', f '',x^2f$ are $L^2$ as well. In particular, $\Theta$ is contained in the dual of this space. 
\end{remark}
}
\subsection{The transform $\L \xi$ of the white noise }

Let $\xi$  be the white noise on $\H$ defined as in \eqref{def:xi}. Since it takes its values in $\Theta$, to define  $\L \xi(z)$ we first need to extend $\L$ to the larger space $\Theta$. To do so, we assume from now on the stronger assumption that, for any compact $K\subset\Lambda$,
\eq
\label{Ass+}
\tag{$C_K$}
C_K:=\sup_{z\in K}\,\sum_{k\in\N} (1+k^2)|\Psi_k(z)|^2<\infty\, .
\qe
Now, for any $\theta\in\Theta$, since $\overline{\H}^{\|\cdot\|_\Theta}=\Theta$ one can find a sequence $(h_n)_{n\in\N}\subset\H$ such that $h_n\to\theta$ in $\Theta.$
Using \eqref{Ass+} and the Cauchy-Schwarz inequality we obtain for any $f\in\H$,
\eq
\label{TzLin}
|\L f(z)|\leq \sqrt{C_K} \|f\|_\Theta
\qe
uniformly for $z\in K,$ and thus $\{ \L h_n\}_{n\in\N}$ is a Cauchy sequence in the space $\mathcal A(\Lambda)$ of analytic functions on $\Lambda$, equipped with the topology of uniform convergence on compact subsets. Since $\mathcal A(\Lambda)$ is complete, this provides the existence of a unique $\L \theta\in\mathcal A(\Lambda)$ such that $\L h_n\to \L\theta$ in $\mathcal A(\Lambda)$, independently on the choice of the sequence $(h_n)_{n\in\N}$ of $\H$ converging to $\theta$. Thus we have obtained a well-defined and unique continuous extension $\L :\Theta\to\mathcal A(\Lambda)$ for the mapping $\L $.

Provided with this extended map, $\L \xi$ is now a well defined random variable taking values in $\mathcal A(\Lambda)$. To describe its law, let us introduce the finite dimensional approximation for any $n\geq 0$,
$$
\xi^{(n)}:=\sum_{k=0}^n \xi_k\,f_k
$$
where the $\xi_k$'s come from the same sequence used to construct $\xi$ in \eqref{def:xi}. Note that by construction $\L  f_k=\Psi_k$  for any $k\in\N$ and thus, for any $z\in\Lambda$,
$$
\L\xi^{(n)}(z)=\sum_{k=0}^n\xi_k\,\Psi_k(z).
$$
Since $\xi^{(n)}\to\xi$ in $\Theta$ a.s.  (in fact, we provide a much more precise result in Proposition \ref{p:concentration} below) and $\L$ is continuous,  we obtain the a.s. convergence in $\mathcal A(\Lambda)$,
\eq
\label{Ttheta}
\L \xi=\lim_{n\to\infty} \L \xi^{(n)}=\sum_{k\in\N} \xi_k \Psi_k\,.
\qe
In particular, the right hand side of \eqref{Ttheta} converges a.s. in $\mathcal A(\Lambda)$ and, for any other random variable $\tilde\xi$ with law $\mu$, we have
$$
\L \tilde\xi \law \sum_{k\in\N} \xi_k \Psi_k\,.
$$

\subsection{Summary and a template theorem}
\label{sec:summary}
We summarize Section~\ref{s:white}, that is, the construction of white noise and its transform, in Proposition~\ref{p:white} and Theorem~\ref{t:main}. Section~\ref{s:proofs} is a list of applications of Theorem~\ref{t:main}. For ease of future reference, we repeat some definitions, and add ``revisited'' to the original labels of the corresponding equations.

\begin{proposition}[Definition of white noise]
\label{p:white}
Let $\H$ be a complex separable Hilbert space, and let $(f_k)$ be an orthonormal basis of $\H$. Let $\|\cdot\|_\Theta$ be defined by
\eq
\tag{\ref{Theta_norm} revisited}
\| f\|_\Theta^2:=\sum_{k\in\N}\frac1{1+k^2}\,\big|\langle f,f_k\rangle\big|^2,
\qe
and let $\Theta$ be the completion of $\H$ for that norm. Then the series
$\xi:=\sum_{k\in\N}\xi_k\,f_k\,,$
with $(\xi_k)_{k\in\N}$ i.i.d $N_\C(0,1)$, converges in $\Theta$ almost surely, and its limit has the characteristic function of white noise \eqref{eqFour}.
\end{proposition}

\begin{theorem}[Transform of white noise]
\label{t:main}
With the notation of Proposition~\ref{p:white}, let further $(\Psi_k)_{k\in\N}$ be a sequence of analytic functions on an open subset $\Lambda\subset\C\cup\{\infty\}$. If \eq
\tag{\ref{Ass+} revisited}
C_K:=\sup_{z\in K}\,\sum_{k\in\N} (1+k^2)|\Psi_k(z)|^2<\infty
\qe
for any compact $K\subset\Lambda$, then the continuous linear mapping $\L:\H\to \mathcal A(\Lambda)$ defined by
\eq
\tag{\ref{L} revisited}
\L f(z):=\sum_{k=0}^\infty \langle f_k,f\rangle \Psi_k(z),\qquad z\in\Lambda,
\qe
extends uniquely to a continuous linear mapping $\L:\Theta\to \mathcal A(\Lambda)$, and the transform of white noise
\eq
\label{e:transformOfWhiteNoise}
\L \xi=\sum_{k\in\N}\xi_k\,\Psi_k
\qe
is well-defined, where the right hand side converges almost surely in $\mathcal A(\Lambda).$ Moreover, by construction,
$$
\L f_k(z)=\Psi_k(z).
$$
\end{theorem}

\begin{remark} All the previous discussion applies to the setting of real Hilbert spaces after replacing any variable $N_\C(0,\sigma^2)$ by $N(0,\sigma^2).$
\end{remark}

\section{Proofs of the main results }
\label{s:proofs}

In this section, we give applications of Theorem~\ref{t:main} for several triplets $(\H, (f_k), (\Psi_k))$. In particular, we show that the transform \eqref{L} spans known time-frequency transforms as well as new ones.

\subsection{(Higher order) Bargmann transform and Hermite polynomials}
\label{e:gabor}
\begin{proof}[Proof of Theorems \ref{t:planarL} and \ref{t:planarLN}] Take $\H:=L^2(\R,\C)$ and $\langle f,g\rangle:=\int_\R \overline f g \,\d x$. Let $H_k(x)$ be the $k$-th Hermite polynomial, namely $H_k$ is of degree $k$, with positive leading coefficient, and
$$
\int_\R H_k(x)H_\ell(x)\e^{-x^2}\d x=\sqrt{\pi}2^k k! \,\delta_{k\ell}.
$$
The Hermite polynomials further have the generating function, for $x\in\R$ and $ z\in\C$,
\eq
\label{genHermite}
\sum_{k=0}^\infty H_k(x) \frac1{k!}z^k=\e^{2xz-z^2},
\qe
see \cite[Section 4.6]{Ism05}. We will actually use a more general generating function: For any $N\geq 0$, $x\in\R$, and $z\in\C$,
\eq
\label{genHermite2}
\sum_{k=0}^\infty H_{k+N}(x) \frac1{k!}z^k=H_N(x-z)\,\e^{2xz-z^2},
\qe
see \cite[Equation 4.6.29]{Ism05}. Consider the orthonormal basis of $\H$
\eq
f_k(x):=\frac{1}{\sqrt{\sqrt{\pi}2^k k!}} \,H_k(x)\e^{-x^2/2},
\label{e:hermiteFunctions}
\qe
and the analytic functions
\eq
\label{e:weylMonomials}
\Psi_k(z)=\Psi_{k,N}(z):=\begin{cases}
\displaystyle\frac{\sqrt{k!}}{(k-N)!} \, z^{k-N}& \text{ if } k\geq N\\
0 & \text{ if } k< N
\end{cases}
\,,
\qe
which satisfy \eqref{Ass+} for $\Lambda:=\C$. Together with \eqref{Lbis}, the identity \eqref{genHermite2} then yields for any $f\in \H$ and $z\in\Lambda$,
  \begin{align*}
  \L f(z) & =\frac1{\pi^{1/4}}\int_\R \overline{f(x)}\left(\sum_{k=0}^\infty   H_{k+N}(x)\frac1{k!}\left(\frac{z}{\sqrt 2}\right)^k\e^{-x^2/2}\right)\d x,\\
  & =\frac1{\pi^{1/4}} \int_\R \overline{f(x)}\,H_N(x-z)\,\e^{-z^2/2+\sqrt2 xz-x^2/2}\d x\\
  & =\L_{\C,N}f(z),
  \end{align*}
  where $\L_{\C,N}$ is the higher order Bargmann transform \eqref{bargN}. Theorem~\ref{t:planarLN} thus follows from 
  $$
  \frac{\d^N}{\d z^N}\Psi_{k,0}(z)=\Psi_{k,N}(z)
  $$
 and   Theorem~\ref{t:main}. Theorem~\ref{t:planarLN} is also obtained as the particular case where $N=0$.
  \end{proof}

  \subsection{A discrete STFT and Charlier polynomials}
\label{sec:proofCharlier}

\begin{proof}[Proof of Theorem~\ref{t:planarl}]
Take $\H:=\ell^2(\N,\C)$ with the usual inner product $\langle f,g\rangle:=\sum_{x\in\N} \overline{f(x)}g(x)$. We denote by $C_k^{(a)}(x)$ be the $k$-th Charlier polynomial, satisfying
\eq
\label{orthoCharlier}
\sum_{x\in\N} C_k^{(a)}(x) C_\ell^{(a)}(x)\; \frac{a^x}{x!}= \frac{k!}{a^k}\e^a\delta_{k\ell}
\qe
with generating function, for $x\in\N$ and $ z\in\C$,
\eq
\label{genCharlier}
\sum_{k=0}^\infty C^{(a)}_k(x) \frac1{k!} z^k=(1-\frac za)^x\e^z,
\qe
see \cite[Equations (6.1.21-22)]{Ism05}. If we take for orthonormal basis the Charlier functions,
\eq
\label{e:Charlierfunction}
f_k(x):= \e^{-a/2}\frac{a^{k/2}}{\sqrt {k!}}\,C_k^{(a)}(x)\, \frac{a^{x/2}}{\sqrt {x!}}
\qe
and for analytic functions \eqref{e:weylMonomials}, then we obtain for any $f\in \H$ and $z\in\Lambda$
\begin{align*}
\L f(z)& = \e^{-a/2}\sum_{x\in\N}\overline{f(x)}\sum_{k=0}^\infty\frac1{k!}\,C_k^{(a)}(x) (\sqrt az)^k  \left(\frac{a^x}{x!}\right)^{1/2}\\
&= \e^{-a/2+\sqrt a z}\sum_{x\in\N}\overline{f(x)}(\sqrt a- z)^x \frac{1}{\sqrt{x!}} = \e^{-a/2+\sqrt a z} \ell_{\C}f(\sqrt a- z).
\end{align*}
and thus
$$
 \ell_{\C}f(z)=\e^{-a/2+\sqrt a z}\L f(\sqrt a -z).
$$
Note that the mapping $z\mapsto \sqrt a-z$ is a planar isometry which is its own inverse. The theorem now follows from Theorem~\ref{t:main}  and that
$$
\e^{-a/2+\sqrt a z} \GAF_\C^{(1)}(\sqrt a -z) \law  \GAF^{(1)}_\C(z),
$$
which is obtained by observing that the covariance kernels of both sides are identical.
\end{proof}

\subsection{Daubechies-Paul wavelet transform and Laguerre polynomials}
\label{s:analyticWavelets}

\begin{proof}[Proof of Theorem \ref{t:Hyp}]

We take here $\H:=H^2(\R)$ the space of analytic signals, fix $\beta>-1/2$ and let $L_k^{(2\beta)}(x)$ be the $k$-th Laguerre polynomial of parameter $2\beta$, satisfying
\eq
\label{Lag}
\int_{\R_+} L^{(2\beta)}_k(x)L^{(2\beta)}_\ell(x)\,x^{2\beta}\e^{-x}\d x=\frac{\Gamma(k+{2\beta}+1)}{k!} \,\delta_{k\ell}.
\qe
with the generating function, for any $z\in\mathbb D$,
\eq
\label{Lag2}
\sum_{k=0}^\infty L^{({2\beta})}_k(x) z^k=\frac1{(1-z)^{{2\beta}+1}}\e^{-x\tfrac{z}{1-z}},\qquad x\in\R_+,
\qe
see \cite[Section 4.6]{Ism05}. Take for basis $f_k$ of $\H^2(\R)$ the inverse Fourier transforms of the Laguerre functions, that is
\eq
\label{LagONB}
\hat f_k(x):= \sqrt{\frac{k!}{\Gamma(k+{2\beta}+1)}}\, L_k^{(2\beta)}(x)\, x^{\beta} \e^{-x/2} \IND_{\mathbb{R}_+}(x),
\qe
and set
\eq
 \Psi_k(z):=\sqrt{\frac{\Gamma(k+2\beta+1)}{k!}}z^k,\qquad z\in\mathbb D\,,
\label{e:hyperbolicMonomials}
\qe
which are easily seen to satisfy Assumption~\eqref{Ass+} for $\Lambda:=\mathbb D$. We then obtain, for any $f\in \H$ and $z\in\Lambda$,
\begin{align*}
\L f(z) & =\sum_{k=0}^\infty\langle f,f_k\rangle \Psi_k(z)\\
& = \sum_{k=0}^\infty\langle \hat f,\hat f_k\rangle \Psi_k(z)\\
& =  \int_{\R_+} \overline{\hat f(x)} \left(\sum_{k=0}^\infty L_k^{(2\beta)}(x) z^k\right) x^{\beta} \e^{-x/2}\d x\\
& = \frac1{(1-z)^{2\beta+1}} \int_{\R_+} \overline{\hat f(x)}\, x^{\beta}\e^{-\frac x 2\tfrac{1+z}{1-z}} \d x\\
&= \L_\D^{(\beta)}f(z),
\end{align*}
by definition \eqref{e:bergmanTransform}. The  first two assertions of the theorem follow using  Theorem~\ref{t:main}.

Next, recalling \eqref{e:DPtransform}, we have
$$
W_0 f(-u,s)\propto \L_\D^{(0)}f(\phi(u+\i s))
$$
where the proportionality symbol hides a non-vanishing analytic term, and where $$\phi:w\mapsto \frac{2w-\i}{2w+\i}$$ is a conformal mapping from $\mathbb{C}_+$ to $\mathbb{D}$. Thus the zeros of $$z=u+\i s\in\C_+\mapsto W_0\xi(-u,s)$$ have the same law as the zeros of $\L_\D^{(0)}\xi$ after taking the transform $\phi$. Now we recall that the zeros of $\L_\D^{(0)}\xi$ form a DPP with kernel the Bergman kernel $K_\D$ of the unit disc, see \eqref{e:bergmanKernel}. Consequently, we obtain by a change of variables that the zeros of $z=u+\i s\mapsto W_0 \xi(-u,s)$ form a DPP with kernel
\eq
\label{calculBerg}
K_\D(\phi(z),\phi(w))|\phi'(z)||\phi'(w)| = K_{\C_+}(z,w)\times \frac{\eta(z)}{\eta(w)},\qquad z,w\in\C_+,
\qe
where $K_{\C_+}$ was introduced in \eqref{e:bergmanKernelC+} and
$$
\eta(z):=\frac{(2z+\i)^2}{|2z+\i|^2}.
$$
Indeed, a complex change of variables yields that the DPP obtained from $K_\D(z,w)$ after applying the transform $\phi$ has kernel $K_\D(\phi(z),\phi(w))|\phi'(w)|^2 $ and one obtains  the left hand side of \eqref{calculBerg} by symmetrization, since multiplying a kernel by a factor $g(z)/g(w)$ for some non-vanishing function $g$ does not change the induced DPP. Now, since $\eta$ does not vanish on $\C_+$, the DPPs induced by the kernels \eqref{calculBerg} and $K_{\C_+}(z,w)$ are the same and the proof is complete.

\end{proof}

\subsection{A discrete STFT and Meixner polynomials}
\label{sec:proofMeixner}

\begin{proof}[Proof of Theorem \ref{t:hyp}] Let $\H:=\ell^2(\N,\C)$ equipped with the usual inner product, and let $\alpha>-1$ and $0<c<1$. We consider the Meixner polynomials satisfying the orthogonality relations
$$
\sum_{k=0}^\infty M_k(x) M_\ell(x)\frac{\Gamma(x+\alpha+1) c^x}{x!}=  \frac{\Gamma(\alpha+1)^2 k!}{c^k(1-c)^{\alpha+1}\Gamma(k+\alpha+1)}\,\delta_{k\ell}\,
$$
for any $k,\ell\in\mathbb{N}$. They satisfy the generating formula \cite[Section 6.1]{Ism05}
$$
\sum_{k=0}^\infty\frac{\Gamma(k+\alpha+1)}{k!}M_k(x) z^k=\Gamma(\alpha+1)\left(1-\frac zc\right)^x(1-z)^{-x-\alpha-1},
$$
for any $z\in\D$. Let us further consider the basis of $\H$ given by the Meixner functions,
\eq
\label{e:meixnerFunction}
f_k(x):=\frac{(1-c)^{(\alpha+1)/2}}{\Gamma(\alpha+1)}\sqrt{\frac{\Gamma(k+\alpha+1) c^k}{k!}} M_k(x) \sqrt{\frac{\Gamma(x+\alpha+1) c^x}{x!}}
\qe
and take for analytic functions on $\Lambda:=\D$,
$$
\Psi_k(z):=\sqrt{\frac{\Gamma(k+\alpha+1) }{k!}} z^k.
$$
This then yields, for any $f\in\H$ and $z\in\D$,
\begin{align*}
\L f(z) &  = \frac{(1-c)^{(\alpha+1)/2}}{(1-\sqrt{c}z)^{\alpha+1}} \sum_{x\in\N} \overline{f(x)} \left(\frac{\sqrt{c}-z}{1-\sqrt{c}z}\right)^x\sqrt{\frac{\Gamma(x+\alpha+1)}{x!}}\\
& = \frac{(1-c)^{(\alpha+1)/2}}{(1-\sqrt{c}z)^{\alpha+1}}\ell_\D^{(\alpha)} f\left(\frac{\sqrt{c}-z}{1-\sqrt{c}z}\right).
\label{e:transfoMeixner}
\end{align*}
Note that the mapping  $z\mapsto (\sqrt{c}-z)/(1-\sqrt{c}z)$ is an isometry of the hyperbolic disk, of the form \eqref{isomD} with  $a:=-\i/(1-c)$ and $b:=\i\sqrt c/(1-c)$, which is its own inverse. Inverting the relation, we obtain
$$
\ell_\D^{(\alpha)} f(z)= \frac{(1-\sqrt{c} z)^{\alpha+1}}{(1-c)^{(\alpha+1)/2}} \L f\left(\frac{\sqrt{c}-z}{1-\sqrt{c}z}\right).
$$
Now, the theorem follows from Theorem~\ref{t:main} and the fact that
$$
 \frac{(1-\sqrt{c}z)^{\alpha+1}}{(1-c)^{(\alpha+1)/2}} \GAF_{\D}^{(\alpha)}\left(\frac{\sqrt{c}-z}{1-\sqrt{c}z}\right)\law  \GAF_{\D}^{(\alpha)}(z).
$$
The latter can be checked directly by computing the covariance function of both Gaussian processes, see also the proof of \cite[Proposition 2.3.4]{HKPV09}.

\end{proof}
\subsection{A discrete STFT and Krawtchouk polynomials}
\label{sec:proofDiscF}

\begin{proof}[Proof of Theorem \ref{t:spherical}] Here we consider the setting where $\H:=\C^{N+1}$ is equipped with the usual inner product, $\langle f,g\rangle:=\sum_{x=0}^N\overline{f(x)}g(x)$.
  Let $p\in(0,1)$ and consider the Krawtchouk polynomials $K_k^{(p)}(x)$ satisfying the orthogonality relations
  $$
  \sum_{x=0}^N K_k^{(p)}(x) K_\ell^{(p)}(x) {N \choose x}p^x(1-p)^{N-x}={N \choose k}^{-1} \left(\frac{1-p}{p}\right)^k \delta_{k\ell}
  $$
  for any $k,\ell\in\mathbb{N}$. They have for generating formula \cite[Section 6.2]{Ism05}
  $$
  \sum_{k=0}^N{N \choose k} K_k^{(p)}(x) z^k=\left(1-\left(\frac{1-p}p\right)z\right)^x(1+z)^{N-x}.
  $$
Let us now use the notation $q:=\sqrt{\frac{1-p}p}$ for convenience.
  These polynomials yield an orthonormal basis for $\C^{N+1}$, the Krawtchouk functions, given by
  \eq
  f_k(x)=\frac{1}{(1+q^2)^{N/2}}\sqrt{{N \choose k} }q^{-k}  K_k^{(p)}(x) \sqrt{{N \choose x} }q^{-x}.
\label{e:fkKrawchouk}
\qe
 We set, for any $z\in\Lambda:=\C$,
$$
\Psi_k(z):= \IND_{k\leq N}\sqrt{{N \choose k}} z^k,
$$ which trivially satisfy Assumption~\eqref{Ass+}, and together with \eqref{e:fkKrawchouk}, this yields for $f\in\H$,
$$
  \L f(z) =\left(\frac{q+z}{\sqrt{1+q^2}}\right)^N\sum_{x=0}^N \overline{f(x)} \left(\frac{1-qz}{q+z}\right)^x\sqrt{ {N \choose x}}=\left(\frac{q+z}{\sqrt{1+q^2}}\right)^N\ell_{\mathbb S}^{(N)}f \left(\frac{1-qz}{q+z}\right).
  $$
  Note that the mapping $z\mapsto (1-qz)/(q+z)$ is an isometry of the sphere of the form \eqref{e:isomSphere} with $a:=\i q/\sqrt{1+q^2}$ and $b:=\i/\sqrt{1+q^2}$
that is its own inverse. Thus,
  $$
  \ell_{\mathbb S}^{(N)}f(z)= \left(\frac{q+z}{\sqrt{1+q^2}}\right)^N \L f\left(\frac{1-qz}{q+z}\right),
  $$
  and the theorem follows from Theorem~\ref{t:main} and that
  $$
  \left(\frac{q+z}{\sqrt{1+q^2}}\right)^N\GAF_{\mathbb S}^{(N)}\left(\frac{1-qz}{q+z}\right) \law \GAF_{\mathbb S}^{(N)}(z),
  $$
  which can be checked by computing the covariance functions of both sides, see also the proof of \cite[Proposition 2.3.4]{HKPV09}. We note that to prove  $\ell_{\mathbb S}^{(N)}f(z)= \GAF_{\mathbb S}^{(N)}(z)$ in law, one can  alternatively use here that the mapping $U:\C^{N+1}\to\C^{N+1}$ defined by $U\delta_{kx}=f_k(x)$ is unitary and that the law of the finite dimensional white noise $\xi$ is invariant under such transformations.

  \end{proof}

\subsection{Analytic projection and trigonometric polynomials}
\label{proof:analyticproj}
\begin{proof}[Proof of Theorem~\ref{th:analyticproj}] We simply apply Theorem~\ref{t:main} with $\H:=L^2([0,2\pi],\C)$, $\Lambda:=\D$, $\Psi_k(z):=z^{k-N}\IND_{k\geq N}$ and $f_k(x):=\e^{\i kx}/\sqrt{2\pi}$ for $k\in\Z$.
\end{proof}

\section{Implementation details and approximation guarantees}
\label{s:computation}

{In this section, we discuss computational approximations that will likely be necessary in signal processing practice. In Section~\ref{s:details}, we describe the heuristics behind Figures~\ref{f:planarDemo}, \ref{f:sphericalDemo}, and \ref{f:hyperbolicDemo}. This illustrates how we numerically locate zeros, and how we numerically compute integral transforms. These approximations are classical. Now, there is one type of approximation for which GAFs give new guarantees. In Section~\ref{s:finiteApproximationWhiteNoise}, we mathematically quantify the loss in accuracy incurred by approximating white noises by the partial sum of their defining series. In Section~\ref{s:finiteApproximationTransform}, we discuss the effect of the same truncation on the zeros of the transform of white noise. For this particular type of approximation, we show that complex analysis yields new theoretical insights that are of direct practical impact. In particular, using only a truncated white noise, we can identify areas of the complex plane where the zeros of the transform of the \emph{full} white noise are bound to be, with large probability. We illustrate the strengths and limitations of such statements on a numerical example for the hyperbolic GAF.}

\subsection{Approximating continuous transforms and locating their zeros}
\label{s:details}

In Figures~\ref{f:planarDemo}, \ref{f:sphericalDemo}, and \ref{f:hyperbolicDemo}, we meant to illustrate time-frequency transforms of white noise by actually implementing the transforms rather than directly plotting truncated GAFs. First, this mimicks signal processing practice, and second, it turns out that spectro- and scalograms of white noise are relatively ``flat'' across time and frequency/scale, which renders locating their zeros relatively easy and independent of the position on the time and frequency/scale axes. On the contrary, truncated GAFs can behave wildly when $\vert z\vert >1$ and computations quickly reach machine precision.

Two important practical points are how to approximate white noises and how to approximate time-frequency transforms that are integrals or infinite sums. Similarly to \citep{BaFlCh18}, we first reduce (and thus often approximate) all transforms to finite discrete Fourier transforms for computational convenience. More precisely, in Figure~\ref{f:planarDemo}, we approximated the Charlier transform by truncating its discrete Fourier form \eqref{e:cmp}, and the white noise was then approximated by a simple white Gaussian vector. The latter approximation is an arbitrary choice, we could also have directly truncated the expansion of the white noise onto the Charlier basis. Figure~\ref{f:sphericalDemo} did not require any approximation, since transform and white noise are finite by definition. Figure~\ref{f:hyperbolicDemo} required more thought. First, the white noise was approximated by truncating the sum $\sum \xi_k f_k$, where $\xi_k$ are independent complex unit Gaussians, and $f_k$ are defined as in Theorem~\ref{t:Hyp}. Second, the continuous wavelet transform was computed by discretizing its Fourier counterpart.

Finally, in all cases, we define numerical zeros as the grid cells that are \emph{(i)} local minima among their 8 neighbours and \emph{(ii)} have a spectrogram value that is less than 10\% of the overall maximum value. This threshold can be tuned so as to make the mean number of zeros across realizations equal to its theoretical value, computed using the intensity of the corresponding GAF.

\subsection{The finite dimensional-approximation of the white noise}
\label{s:finiteApproximationWhiteNoise}
For practical reasons, it is not enough to abstractly define the white noise $\xi$ as we did in Section~\ref{s:white}: we need to show that finite dimensional approximations, which can be implemented, accurately approximate the white noise. More precisely, given a family $(\xi_k)_{k\in\N}$ of i.i.d. $N_\C(0,1)$ random variables, set
\eq
\label{xin}
\xi^{(n)}:=\sum_{k=0}^n\xi_k \,f_k\,, \qquad \xi:=\sum_{k=0}^\infty\xi_k \,f_k\, .
\qe
As we explained in Section~\ref{s:white}, the series defining $\xi$ converges a.s. in the space $\Theta$ associated with the basis $(f_k)$ of $\H$, and we denote the distribution of the white noise $\xi$ as $\mu$.

First, we provide an easy estimate for the Wasserstein distance $\mathcal{W}_2$ between the law $\mu^{(n)}$ of $\xi^{(n)}$ and $\mu$. Recall that for any $p\geq 1$ the Wasserstein-Kantorovich distance $\mathcal{W}_p(\nu_1,\mu_2)$ between two probability measures on $\Theta$ is defined by
$$
\mathcal{W}_p(\nu_1,\nu_2)^p:=\inf_\pi \iint \|f-g\|_\Theta^p\d \pi(f,g)=\inf_{X,Y}\E\|X-Y\|_\Theta^p\, ,
$$
where the first infimum is taken over all probability measures $\pi$ on $\Theta\times\Theta$ having for marginals $\nu_1$ and $\nu_2$, whereas the second infimum is taken over all couples $(X,Y)$ of random variables with respective marginal distributions $\nu_1$ and $\nu_2$. Convergence in $\mathcal{W}_p$ is equivalent to weak convergence plus convergence of the first $p$-moments.

\begin{proposition} For any $n\geq 1$,
$$
\mathcal{W}_2(\mu^{(n)},\mu)\leq \frac1{\sqrt n}\,.
$$
\end{proposition}
\begin{proof}
By definition of $\mathcal{W}_2$ and Fubini-Tonelli's theorem,
$$
\mathcal{W}_2(\mu^{(n)},\mu)^2\leq\E\|\xi^{(n)}-\xi\|_\Theta^2=\E\sum_{k=n+1}^\infty \frac{|\xi_k|^2}{1+k^2}=\sum_{k=n+1}^\infty \frac{1}{1+k^2}\leq \frac1n\,.
$$
\end{proof}
The Kantorovich-Rubinstein dual representation for $\mathcal{W}_1$ states that, if we denote by  $\|F\|_{\mathrm{Lip}}$ the Lipschitz constant of a Lipschitz function $F:\Theta\to\R$, then we have
$$
\mathcal{W}_1(\nu_1,\nu_2)=\sup_{\|F\|_{\mathrm{Lip}}\leq 1}\int F\,\d\nu_1-\int F\,\d\nu_2.
$$
Since the Cauchy-Schwarz inequality provides $\mathcal{W}_1(\nu_1,\nu_2)\leq \mathcal{W}_2(\nu_1,\nu_2)$, one  can replace
$\mathcal{W}_2$ by $\mathcal{W}_1$ in the previous proposition so as to obtain:

\begin{corollary}
\label{Lipcontrol}
For any Lipschitz function $F:\Theta\to\R$ and $n\geq 1$, we have
$$
|\E F(\xi^{(n)})-\E F(\xi)|\leq \frac{\|F\|_{\mathrm{Lip}}}{\sqrt n}\,.
$$
\end{corollary}

The next result provides a precise probabilistic estimate on the deviation of $\xi^{(n)}$ from $\xi$ in~$\Theta$.

\begin{proposition}
\label{p:concentration}For any $n\geq 1$ and $\epsilon>0$, we have
$$
\P\big(\|\xi^{(n)}-\xi\|_\Theta>\epsilon\big)\leq 2\,\e^{-n\epsilon^2/2}.
$$
\end{proposition}

Notice that is is not possible to compare $\xi^{(n)}$ and $\xi$ in $\H$ since $\|\xi^{(n)}-\xi\|=+\infty$ a.s. Let us also mention that we did not try to optimize the constants $2$ and $1/2$, but we believe that the regime $n\epsilon^2 $ is of  optimal order.

\begin{proof} We set
$$
X_n:=\|\xi^{(n)}-\xi\|_\Theta^2=\sum_{k=n+1}^\infty \frac{|\xi_k|^2}{1+k^2}
$$
and compute its Laplace transform. Using independence and that $|\xi_k|^2\law(N^2+M^2)/2$ where $N,M$ are i.i.d. standard $N(0,1)$ variables, we obtain for any $|t|<1+(n+1)^2$,
$$
\E[\e^{t X_n}]  = \prod_{k=n+1}^\infty \E\left[\e^{\frac t{2(1+k^2)}N^2}\right]^2= \prod_{k=n+1}^\infty \frac1{1-\frac{t}{1+k^2}}\,.
$$
As a consequence, and using that $\log(1-x)\geq -x$ for $x\in(0,1/2)$,  we have for  $0<t<(1+(n+1)^2)/2$,
\begin{align*}
\P\big(\|\xi^{(n)}-\xi\|_\Theta>\epsilon\big) & =\P(X_n>\epsilon^2)\\
& \leq \E[\e^{t X_n}]\,\e^{-t\epsilon^2}\\
&\leq \exp\Big\{t\sum_{k=n+1}^\infty \frac1{1+k^2}\Big\}\,\e^{-t\epsilon^2}\\
& \leq \e^{t/n -t\epsilon^2}.
\end{align*}
The result follows by taking $t=n/2$ since $\sqrt\e< 2$.
\end{proof}
In particular this estimate allows us to provide a quantitative control on the zeros of finite dimensional approximations of the transform $\L\xi$ in the next section.

\subsection{The finite-dimensional approximation of $\L \xi$}
\label{s:finiteApproximationTransform}

Recalling \eqref{xin}, we investigate here how close $\L \xi^{(n)}$ is from $\L\xi$. Corollary \ref{Lipcontrol} already provides the next result.

\begin{corollary}
For any compact $K\subset\Lambda$ and any Lipschitz function $\phi:\C\to\R$, we have
$$
\sup_{z\in K}\left|\E(\phi(\L \xi^{(n)}(z)))-\E(\phi(\L \xi(z)))\right|\leq \|\phi\|_{\mathrm{Lip}} \sqrt{ \frac{C_K}{n}}.
$$
\end{corollary}

\begin{proof} Indeed, for any fixed $z\in K$, the mapping $f\mapsto \phi(\L f(z))$ from $\Theta$ to $\R$ is Lipschitz with constant $\|\phi\|_{\mathrm{Lip}}\sqrt{C_k}$ according to \eqref{TzLin} as a composition of Lipschitz functions.
\end{proof}

Finally, we provide a control on the zeros of $\L\xi^{(n)}$, that we observe on numerical simulations like Figure~\ref{f:gafs}, and those of $\L\xi$ which have for distribution the law of the zeros of various GAFs of interest.

\begin{theorem} Let $\epsilon>0$,  fix a compact window $K\subset\Lambda$ and let $n\geq 1$ be such that
\eq
\label{e:assCK}
\sup_{z\in K}\sum_{k=n+1}^\infty(1+k^2)|\Psi_k(z)|^2\leq 1,
\qe
whose existence is guaranteed by Assumption~\eqref{Ass+}. The following holds true with probability at least $1-2\e^{-n\epsilon^2/2}$: given any simple $C^1$ closed curve $\gamma\subset K$ such that $|\L \xi^{(n)}(z)|>\epsilon$  for all $z\in\gamma$, $\L \xi^{(n)}$ and $\L \xi$ have the same number of zeros in the simply connected domain $D$  defined by having boundary $\gamma$, namely $\partial D=\gamma$.
\label{t:rouche}
\end{theorem}

\begin{proof} Let $A$ be the set of events where $\|\xi^{(n)}-\xi\|_{\Theta}\leq\epsilon$ is true. Thus $\P(A)\geq 1-2\e^{-n\epsilon^2/2}$ by Proposition~\ref{p:concentration}. Restricting to the events of $A$, we have by the Cauchy-Schwarz inequality and \eqref{e:assCK} that,  for every $\gamma$ as in the theorem and for every $z\in \gamma$,
$$
|\L (\xi^{(n)}-\xi)(z)| \leq \sqrt{\sum_{k={n+1}}^\infty (1+k^2)|\Psi_k(z)|^2}\; {\| \xi^{(n)}-\xi\|}_{\Theta}\leq \epsilon.
$$
Since by assumption $|\L\xi^{(n)}(z)|>\epsilon$ for every $z\in\gamma$, the result follows from Rouché's theorem.
\end{proof}

{
As an illustration, consider the setting of Section~\ref{sec:hypGAF} where $\H$ is the space of analytic signals $H^2(\R)$. Set the parameter $\beta$ to zero and let $\L=\L_\D^{(0)}$ be the Bergman transform on the disk; thus $\Lambda=\D$ and $\Psi_k(z)=z^k$. If we take $K=\{z\in\C:\;|z|\leq\delta\}$ for some $0<\delta<1$, it follows from rough upper bounds
that Assumption~\eqref{e:assCK} is satisfied as soon as
$$
n\geq \max\left(\frac{\log((1-\delta^2)/2)}{\log\delta},\frac4{\log^2\delta}\right).
$$
For instance if $\delta=0.95$ then \eqref{e:assCK} holds as soon as $n\geq 1520$. Now, take $\epsilon=0.1$, so that $1-2\e^{-1520\epsilon^2/2}\geq 0.99$. Given a realization $(\xi_k(\omega))_{k\in\N}$ of the i.i.d standard complex Gaussian variables, if the polynomial
$$
\mathscr L_\D^{(0)} \xi^{(1520)}_\omega(z):=\sum_{k=0}^{1520}\xi_k(\omega)\, z^k
$$
has $m$ zeros inside any domain $D_\gamma$ delimited by a simple closed contour $\gamma\subset K$ such that  $|\mathscr L_\D^{(0)} \xi^{(1520)}_\omega(z)|>0.1$ for every $z\in\gamma$, then with more than $99\%$ chance the associated realization of the full transform
$$
\mathscr L_\D^{(0)} \xi_\omega(z):=\sum_{k=0}^{\infty}\xi_k(\omega)\, z^k
$$
has $m$ zeros  in $D_\gamma$ as well.
}

{Figure~\ref{f:gruyere} illustrates a realization of the truncated hyperbolic GAF with these values for $\epsilon$, $n$ and $\delta$. The level lines $\vert z\vert =\epsilon = 0.1$ are shown in red. Since each of the detected 11 zeros is alone in its red contour, we can be confident that the limiting GAF only has one zero inside each red contour as well. However, while statements like Theorem~\ref{t:rouche} provide new insights tools, an important limitation is that the theorem does not provide any information on the zeros in the annulus $\D\setminus K$, while most of the $1520$ zeros of the polynomial plotted in Figure~\ref{f:gruyere} are actually there, i.e., in $\vert z\vert >\delta$.
}


\begin{figure}
  \centering
  \includegraphics[width=.65\textwidth]{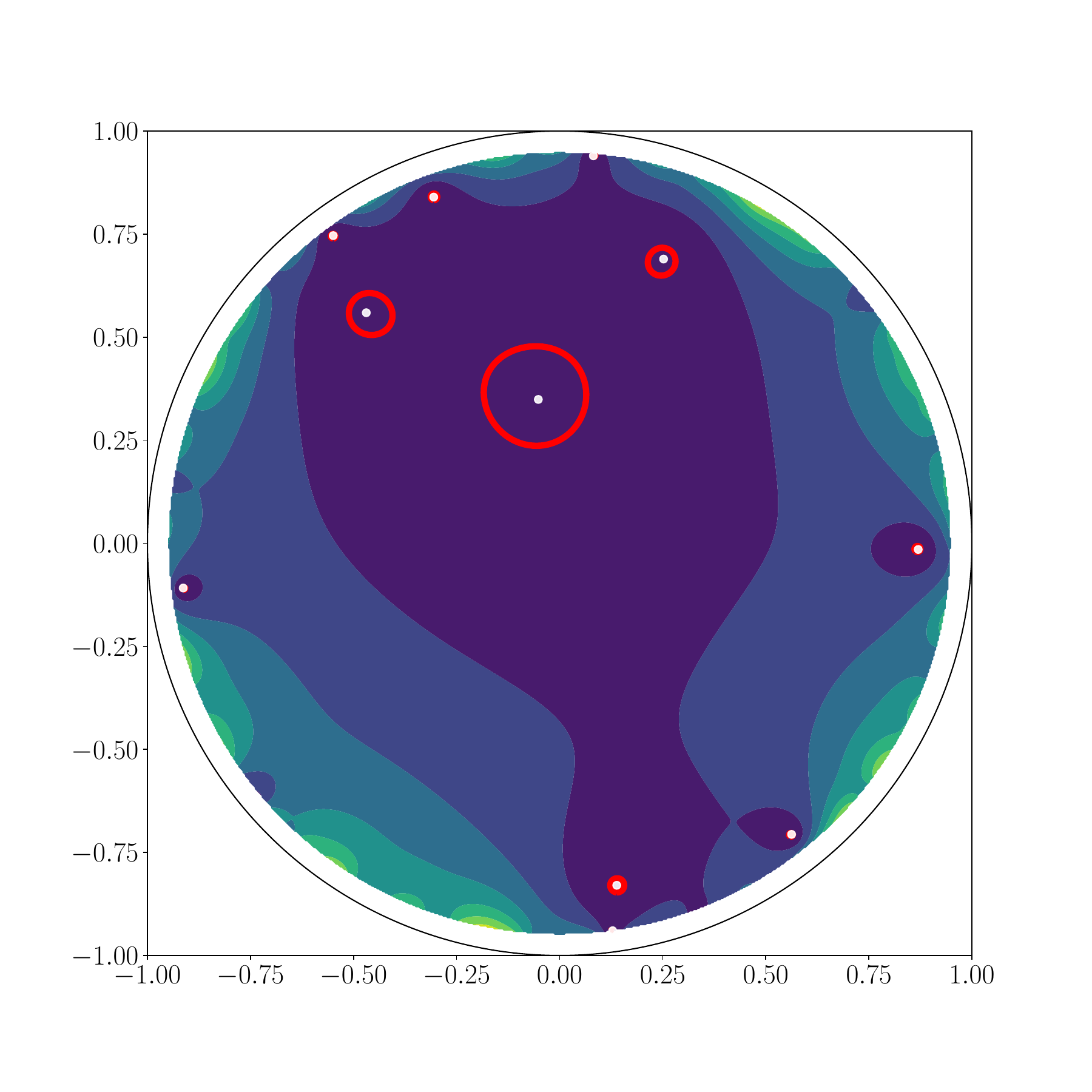}
  \caption{Illustrating how to use Theorem~\ref{t:rouche} to guarantee the location of some of the zeros of the hyperbolic GAF.}
  \label{f:gruyere}
\end{figure}

\section{Discussion}
\label{s:discussion}
We have formalized a correspondence between signal processing transforms and GAFs through generating identities for orthogonal polynomials. This correspondence allows to easily identify the law of the zeros of many time-frequency transforms of white noises, which is a prerequisite to developing the research line initiated by \cite{Fla15} on signal filtering and reconstruction with zeros of spectrograms and scalograms.

The link with GAFs is important for both signal processing and probability. For signal processing, it provides a toolbox containing many probabilistic results on the zeros of GAFs \citep{HKPV09}; these results can in turn be of statistical use for signal detection and reconstruction. For probabilists, the link with windowed discrete Fourier transforms gives efficient approximate sampling of all three invariant zero sets of GAFs. Moreover, signal processing can raise new natural questions on GAFs, such as understanding the behavior of the zeros when simple signals are present on top of the noise.

Although a systematic treatment of reconstruction procedures with zeros is beyond the scope of this paper, we first stress that all three invariant sets of zeros of GAFs can lead to efficient estimation of summary functionals precisely because of this invariance, so that the frequentist testing procedures of \cite{BaFlCh18} can always be implemented. Second, the discrete transforms that were unveiled in this paper remove one step of approximation in practice, namely the discretization of integrals. Furthermore, we have shown how the remaining approximation step --~the truncation of the sums defining the white noise and its transforms~-- can be controlled.

From a mathematical perspective, it would be interesting to go beyond classical orthogonal polynomials. The results in Section~\ref{s:white} do not depend on the polynomial character of the basis, we merely used polynomials to obtain generating functions. Furthermore, it would be interesting to know how much of our results we can preserve while bypassing the need for analytic-valued transforms. It is not obvious how to preserve the isolation of zeros, for starters, but allowing non-Gaussian windows in the short-time Fourier transform and non-Cauchy wavelets comes at the price of abandoning analyticity \citep{AsBr09,HKPA19}.

\bibliographystyle{plainnat}
\bibliography{stft,stft2,stats}

\end{document}